\documentclass[a4paper,twoside,11pt]{article}  % draft: Zeigt überfüllte Zeilen an

\usepackage[english]{babel}  
\usepackage[utf8]{inputenc}
\usepackage[T1]{fontenc}
\usepackage{lmodern}
\usepackage{anyfontsize}     % Keine Fehlermeldung "Fontshape .. not available" für mathmode in chapter,section,..
\usepackage{latexsym}
\usepackage{amsmath,amssymb,amsthm,bm,enumitem} % enumitem für etwa \begin{enumerate}[label=\alph*)]
\usepackage{mathrsfs}        % Für \mathscr, etwa \mathscr{C}^\infty
\usepackage{mathtools}       % Für \psmallmatrix
\usepackage{nicefrac}        % Für schräge Bruchstriche
\usepackage{marvosym}        % Für Widerspruch \Lightning,\text{\Lightning} in Mathe-Umgebung

\usepackage[inner=25mm,outer=25mm,bottom=20mm,top=32mm,includefoot]{geometry} % Seitenränder
\usepackage{setspace}            % Zeilenabstand 1.5cm
\usepackage{fancyhdr}
\usepackage[dvipsnames]{xcolor} %  dvipsnames für mehr Farben, vor \usepackage{pdfpages} laden
\usepackage{graphicx} % Für etwa \includegraphics[width=0.2\paperwidth,height=0.1\paperheight]{Dateiname}
\usepackage{floatflt} % Text kann Bild umfließen, \begin{floatingfigure}
\usepackage{caption}  % Ermöglicht nur "Abb.4.3" ohne Doppelpunkt mit \caption{}
\usepackage{float}    % Ermöglicht das [H] in \begin{figure}[H], d.h. die Grafik soll genau hier sein
\usepackage{pdfpages} % Einbinden von PDF Dateien oder Seiten daraus, etwa \includepdf[pages=-]{Name der Datei}
\usepackage{tikz-cd}  % Für (kommutative) Diagramme und Zeichnungen
 \usetikzlibrary{arrows.meta,calc,decorations.markings,math}   % Größe des Pfeilkopfes ändern, z.B. \draw[-{Latex[scale=1.2]}]    
\usepackage{forest}   % Für Bäume zeichen mit TikZ, etwa \begin{forest} 
\usepackage{array}    % Bei tabular Länge und Größe bestimmen 
\usepackage{tasks}    % \begin{tasks}(3) \task

\usepackage{makeidx}     % Auf Tools und Makeindex klicken
 \makeindex  
\usepackage[refpage,german]{nomencl} % Symbolverzeichnis, refpage für Seitenzahl, ohne german hat man Nomenclature
 \makenomenclature
\usepackage{hyperref}  % Spezielle PDF Erweiterungen aktivieren (z.B. ein anklickbares Inhaltsverzeichnis), Links und Verweise innerhalb von PDF
 \hypersetup{colorlinks=true,citecolor=Green,linkcolor=blue,urlcolor=Blue}

\usepackage{tocbasic}
 \DeclareTOCStyleEntry[beforeskip=.7em plus 1pt,pagenumberformat=\textbf]{tocline}{section} % kleinerer Abstand zwischen section im Inhaltverzeichnis, Default ist 1em plus 1pt
\usepackage{sectsty}               % Für chapter etc. mittig, die Größe, etc.
 \sectionfont{\centering}
 \subsectionfont{\fontsize{13}{15}\selectfont}
\numberwithin{equation}{section}   % Nummeriert die Gleichungen mit der Section-Nummer, z.B. in chapter 9 section 5, 3te Gleichung dort, 114te insgesamt, ist dann (9.5.3) statt (9.114). 
\allowdisplaybreaks                % Erlaubt Seitenumbruch innerhalb von align 
\theoremstyle{plain}
\newtheorem{Theorem}{Theorem}[subsection]

\newtheorem{Lemma}[Theorem]{Lemma}

\newtheorem{innercustomthm}{Theorem}
\newenvironment{customthm}[1]{\renewcommand\theinnercustomthm{#1}\innercustomthm}{\endinnercustomthm} 

\theoremstyle{definition}
\newtheorem{Definition}[Theorem]{Definition}

\theoremstyle{remark}
\newtheorem{Remark}[Theorem]{Remark}

\newcommand{\N}{\mathbf{N}}   % Natürliche 
\newcommand{\Z}{\mathbf{Z}}   % Ganze
\newcommand{\R}{\mathbf{R}}   % Reelle
\newcommand{\C}{\mathbf{C}}   % Komplexe
\newcommand{\Cl}{\mathrm{Cl}}

\newcommand{\ve}{\varepsilon}
\newcommand{\vp}{\varphi}
\newcommand{\vr}{\varrho}
\newcommand{\vt}{\vartheta} 
\newcommand{\dist}{\mathrm{dist}} % Distanz
\newcommand{\inj}{\mathrm{inj}}   % Injektivitätsradius
\newcommand{\dvol}{d\mathrm{vol}}
\newcommand{\pr}{\mathrm{pr}}

\newcommand{\surjright}[1]{\xrightarrow{\hspace{#1}}
\kern-1.9ex\xrightarrow{}} %Pfeil für Surjektion/Projektion, etwa \surjright{0cm}

\DeclareMathOperator{\id}{id}
\DeclareMathOperator{\Tr}{Tr}        % Spur bzw. Trace als Befehl
\DeclareMathOperator{\Sp}{Sp}        % Spektrum als Befehlt
\DeclareMathOperator{\rk}{rk}        % Rang
 
\DeclareMathOperator{\End}{End}

\DeclareMathOperator{\Td}{Td}        % Charakteristische Klassen
\DeclareMathOperator{\ch}{ch}
\DeclareMathOperator{\ctop}{\mathit{c}_{top}} % top Chern class

\def\Re{\mathrm{Re}\,}

\title{Asymptotic equivariant holomorphic torsion forms}
\author{Pascal Teßmer}
\date{}

\begin{document}

\maketitle
\begin{abstract}
\noindent In this paper we study the asymptotic behaviour of the equivariant holomorphic analytic torsion forms associated with increasing powers of a fibrewise positive line bundle. The result is an equivariant extension of a result of Puchol.
\end{abstract}

\tableofcontents

{\let\thefootnote\relax\footnote{\textit{Date}\,: \today.}
\let\thefootnote\relax\footnote{2020 \textit{Mathematics Subject Classification}. 41A60, 58J20, 58J52.}
\let\thefootnote\relax\footnote{\textit{Key words and phrases}. analytic torsion; asymptotics.}}

\addcontentsline{toc}{section}{Introduction}
\section*{Introduction}
The holomorphic analytic torsion was introduced in 1973 by Ray and Singer (\cite{RSi73}) as an analogue of the real analytic torsion. It is a positive real number associated with the spectrum of the Kodaira Laplacian of  holomorphic vector bundles on a compact complex manifold. In the works of Bismut-Gillet-Soulé (\cite{BGS88a,BGS88b,BGS88c}) regarding derterminant bundles, the authors showed that it gives rise to a metric with desired properties. Another application lies in Arakelov geometry where Gillet-Soulé proved an arithmetic Grothendieck-Riemann-Roch Theorem in which the holomorphic analytic torsion appears (\cite{GS92}). In \cite{BVa89}, Bismut and Vasserot studied the asymptotic of the holomorphic torsion associated with increasing powers of a positive line bundle and extended it in \cite{BVa90} by replacing the line bundle with symmetric powers of a Griffiths-positive vector bundle. In context of Arakelov geometry, Gillet-Soulé used the asymptotic expansion for a result on arithmetic ampleness (\cite{GS92}). The result of \cite{BVa89} has been sharpened and generalized by Finski in \cite{F18} where he proved a formula for the full asymptotic of the holomorphic analytic torsion and extended it to orbifolds. 

In \cite{BK92}, Bismut and Köhler introduced the holomorphic analytic torsion form which is an extension of the holomorphic analytic torsion to the family setting. With the provided analytical tool of the holomorphic analytic torsion form and its properties the arithmetic Grothendieck-Riemann-Roch Theorem for higher degrees was established, see \cite{GRS08}. 

In \cite{P23}, Puchol gave an asymptotic formula for the holomorphic analytic torsion forms which generalizes the result of Bismut-Vasserot \cite{BVa89,BVa90} for the family setting as well as for more general bundles. 

The holomorphic analytic torsion has an equivariant version introduced in \cite{K93}. In \cite{KR01}, Köhler-Rössler used the equivariant holomorphic analytic torsion in their work on a Lefschetz type fixed point formula in equivariant Arakelov geometry. Similary, there is a generalization for the family case called equivariant holomorphic analytic torsion form introduced in \cite{M00}. For its applications in Arakelov Geometry, see for instance \cite{K05}. 

Our result concerns the asymptotic of the equivariant holomorphic torsion form and extending certain results of \cite{P23} to the equivariant case. Let us describe it in more detail: Let $\pi:M\rightarrow B$ be a holomorphic fibre bundle with compact fibre $Z$. Let $\omega^M$ be a real $(1,1)$-form on $M$ such that $(\pi,\omega^M)$ defines a Hermitian fibration (cf. Def \ref{Hermitian fibration}). Let $(\mathcal{E},h^\mathcal{E})$ resp. $(\mathcal{L},h^\mathcal{L})$ be a Hermitian holomorphic vector resp. line  bundle on $M$. Let $\Omega^\mathcal{L}$ be the curvature of the Chern connection of $(\mathcal{L},h^\mathcal{L})$.\bigbreak

\noindent\textbf{Assumption.}  We make the assumption that $i\Omega^\mathcal{L}$ is positive along the fibres of $\pi$ (cf. (\ref{pos})).\bigbreak

\noindent For $p\in\N_0$, put $\mathcal{L}^p:=\mathcal{L}^{\otimes p}$. If $B$ is compact, then by the assumption and Kodaira vanishing theorem, we can find $p_0$ such that for $p\ge p_0$, the direct image $R^\bullet\pi_*(\mathcal{E}\otimes\mathcal{L}^p)$ is locally free and vanishes in positive degrees. Now for not necessarily compact $B$, we will assume that such $p_0$ can be chosen uniformly on the compact subsets of $B$. The results in this paper will \textbf{hold for} $\bm{p\ge p_0}$. Let $\gamma$ be an holomorphic isometry of $(M,\omega^M)$ which induces equivariant structures on $(\mathcal{E},h^\mathcal{E}),(\mathcal{L},h^\mathcal{L})$ (cf. Def. \ref{equiv struc}). Let $h^{\mathcal{E}\otimes\mathcal{L}^p}$ be the Hermitian metric on $\mathcal{E}\otimes\mathcal{L}^p$ induced by $h^\mathcal{E}$ and $h^\mathcal{L}$. Let $T_\gamma(\omega^M,h^{\mathcal{E}\otimes\mathcal{L}^p})$ be the associated equivariant holomorphic torsion form of \cite{B13}. Let $Z^\gamma$ denote the fixed-point set of $\gamma$ on $Z$ with connected components $\{Z^\gamma_l\}_{l\in L}$. Put $n_\gamma=\max\{\dim_\C Z^\gamma_l\mid l\in L\}$ and $L_0=\{l\in L\mid \dim_\C Z^\gamma_l=n_\gamma\}$.
The action of $\gamma^\mathcal{L}$ on $\mathcal{L}$ is given by multiplication with $e^{i\theta}$ which is locally constant and hence constant on the $Z^\gamma_l$. For a differential form $\alpha$ on $B$, we denote by $\alpha^{(k)}$ its component of degree $k$. By the term \textit{locally computable}, we will mean quantities which can be expressed by integration along the fibre of forms on $M$ which are locally computable from $\omega^M,h^\mathcal{E},h^\mathcal{L}$. The asymptotic behaviour of the equivariant torsion forms for $p\rightarrow\infty$ is given as follows:
\begin{customthm}{1}\label{Th1}
There exist locally computable differential forms  $\alpha_{\gamma,l},\beta_{\gamma,l}$ on $B$ such that for $0\le k\le\dim_\C B$, as $p\rightarrow\infty$,
\[T_\gamma(\omega^M,h^{\mathcal{E}\otimes\mathcal{L}^p})^{(2k)}=p^{n_\gamma+k}\sum_{l\in L_0}e^{ip\theta}_{|_{Z_l^\gamma}}\left(\alpha_{\gamma,l}\log p+\beta_{\gamma,l}\right)^{(2k)}+o(p^{n_\gamma+k})\]
in the topology of $C^\infty$ convergence on compact subsets of $B$.
\end{customthm}

\noindent\textbf{Acknowledgement.} Some results of this paper were part of the author's PhD thesis at the Heinrich Heine University Düsseldorf. I am deeply indebted to my supervisor Professor Dr. K. Köhler for his support.

\section{Notations and Preliminaries}\label{sec 0}
-\textbf{\textit{General notations.}} For the natural numbers, we use the notations $\N=\{1,2,\ldots\}$, $\N_0=\N\cup\{0\}$. Let $Z$ and $E$ be smooth manifolds and $\pi_E:E\rightarrow Z$ a vector bundle. The space of smooth sections will be denoted by $\Gamma(Z,E)$ and, if not stated otherwise, all sections shall be considered as smooth. The interior product with a vector field $X\in\Gamma(Z,TZ)$ is denoted by $\iota_X$. Now let $Z$ be a complex manifold. We will denote by $T_\R Z$ the real tangent bundle and by $T'Z$ the holomorphic tangent bundle. Set $T_\C Z:=T_\R Z\otimes_\R\C$. Let $J^{T_\R Z}$ be the complex structure on $T_\R Z$ and $J^{T_\C Z}$ its complex linear extension to $T_\C Z$. Let $T^{1,0}Z$ and $T^{0,1}Z$ be the $i$ and $-i$-eigenbundles of $J^{T_\C Z}$. For any $X\in\Gamma(M,T_\C Z)$, we write $X=X^{1,0}+X^{0,1}$ according to the decomposition $T_\C Z=T^{1,0}Z\oplus T^{0,1}Z$. For a complex vector bundle $\pi_E:E\rightarrow Z$, we use the following notations regarding complex differential forms:
\begin{align*}
&{\mathfrak{A}}^k(Z,E):=\Gamma(Z,\Lambda^k T^*_\C Z\otimes E),\quad\mathfrak{A}^\bullet(Z,E):=\bigoplus_{k\ge0}\mathfrak{A}^k(Z,E)&(k\in\N_0),\\
&\mathfrak{A}^{p,q}(Z,E):=\Gamma(Z,\Lambda^p(T^{1,0}Z)^*\wedge\Lambda^q (T^{0,1}Z )^*\otimes E)&(p,q\in\N_0).
\end{align*}
The wedge product will often be omitted where we write $\alpha\beta,\omega^k$ instead of $\alpha\wedge\beta,\omega^{\wedge k}$. The underlying real bundle of a complex vector bundle shall be noted with a subscript $\R$, e.g. $E_\R$. If $(Z,g^{T_\R Z})$ is a Riemannian manifold, $\nabla^{TZ,LC}$ will denote the Levi-Civita connection. The Riemannian volume form will be denoted by $\dvol_{g^{T_\R Z}}$. If $P$ is a smoothing operator acting on $\Gamma(Z,E)$ given by a kernel $p$ with respect to the Riemannian volume form $\dvol_{g^{T_\R Z}}$, then we will use Dirac's notation $\big\langle\cdot\bigm|P\bigm|\cdot\big\rangle$ for the kernel: For $x\in Z$ and $s\in\Gamma(Z,E)$,
\[(Ps)(x)=\int_{y\in Z}\bigl\langle x\bigm|P\bigm|y\bigr\rangle s(y)\,\dvol_{g^{T_\R Z}}(y):=\int_{y\in Z}p(x,y)s(y)\,\dvol_{g^{T_\R Z}}(y).\]
The curvature of a connection $\nabla^E$ will be denoted by $\Omega^E$. If $\nabla^E$ is a Hermitian connection on $(E,h^E)$, the $C^m$-norm on $\Gamma(Z,E)$ induced by $\nabla^E$ and $g^{T_\R Z}$ will be denoted by $\|\cdot\|_{C^m(Z,E)}$ or often just $\|\cdot\|_{C^m(Z)}$ when the range is obvious.\\[0.5em]
\indent-\textbf{\textit{Superspaces.}} We use Quillen's formalism of superbundles (\cite{Q85}). Let $V$ be a finite dimensional $\Z_2$-graded vector space and $\tau$ the involution defining the $\Z_2$-grading. For $A\in\End(V)$, its supertrace will be denoted by 
\[\Tr_s[A]:=\Tr[\tau A].\]
Let $\mathfrak{A},\mathfrak{B}$ be two $\Z_2$-graded algebras. In the whole paper, $\mathfrak{A}\otimes\mathfrak{B}$ will denote the $\Z_2$-graded tensor product as ungraded spaces will be equipped with the trivial superstructure ($\deg v=0~\forall v\in V\setminus\{0\}$). The supertrace extends to $\mathfrak{A}\otimes\End(V)$ by 
\[\Tr_s[a\otimes A]:=a\Tr_s[A]\qquad(a\in\mathfrak{A},~A\in\End(V)).\]
Occasionally, we write $\Tr_s^V$ to emphasize the space. If $H$ is a $\Z_2$-graded Hilbert space, $\tau$ the involution defining its $\Z_2$-grading and $A\in\End(H)$ is trace class, then as before $\Tr_s[A]:=\Tr[\tau A]$.\\[0.5em]
\indent-\textbf{\textit{Family bundles and superconnections.}} Let $M,B$ be complex manifolds and $\pi:M\rightarrow B$ a holomorphic fibre bundle with compact fibre $Z$. Let $\mathcal{E}$ be a complex manifold and $\pi_\mathcal{E}:\mathcal{E}\rightarrow M$ a holomorphic fibre bundle over $M$. For $b\in B$, set $Z_b:={\pi^{-1}\{b\}}$ and $\mathcal{E}_b:=\mathcal{E}_{|_{Z_{b}}}$. We call $\{{\mathcal{E}_b}\}_{b\in B}$ a \textit{family of vector bundles} if $\mathcal{E}_b$ is a vector bundle over $Z_b$ for each $b\in B$. To a given family of vector bundles $\{\mathcal{E}_b\}_{b\in B}$ let $\{\Gamma(Z_b,\mathcal{E}_b)\}_{b\in B}$ be the infinite dimensional vector bundle on $B$ with fibres $\Gamma(Z_b,\mathcal{E}_b)$. Then one has the natural identification
\[\Gamma(B,\Lambda^\bullet T^*_\C B\otimes\{\Gamma(Z_b,\mathcal{E}_b)\}_{b\in B})=\Gamma(M,\pi^*\Lambda^\bullet T^*_\C B\otimes\mathcal{E})\]
which will be used throughout this paper.
A differential operator $A$ acting on $\Gamma(M,\pi^*\Lambda^\bullet T^*_\C B\otimes\mathcal{E})$ is called a \textit{superconnection} if it is of odd parity and satisfies the Leibniz's rule 
\[A(\alpha\wedge\beta)=d^B\alpha\wedge\beta+(-1)^k\alpha\wedge A\beta\qquad(\alpha\in\mathfrak{A}^k(B),~\beta\in\mathfrak{A}^\bullet(B,\{\Gamma(Z_b,\mathcal{E}_b)\}_{b\in B})).\]
It is of the form $A=\sum\limits_{k=0}^{\dim_\R B}A^{(k)}$ with $A^{(k)}:\mathfrak{A}^{\cdot}(B,\{\Gamma(Z_b,\mathcal{E}_b)\}_{b\in B})\rightarrow\mathfrak{A}^{\cdot+k}(B,\{\Gamma(Z_b,\mathcal{E}_b)\}_{b\in B})$. If $D =\{D_b\}_{b\in B}$ is a family of Dirac operators acting on $\{\Gamma(Z_b,\mathcal{E}_b)\}_{b\in B}$, then $A$ is said to be \textit{adapted to} $D$ provided that $A^{(0)}=D$. $A$ will be called \textit{elliptic} if $D^2$ is required to be fibrewise elliptic of order 2 instead of being a generalized Laplacian. The curvature $A^2$ lies in $\mathfrak{A}^\bullet(B,\{\End(\Gamma(Z_b,\mathcal{E}_b))\}_{b\in B})$ since it supercommutes with $\mathfrak{A}^\bullet(B)$ and it has a decomposition $A^2= D^2+A^{2,(+)} $ where $A^{2,(+)}$ raises the exterior degree in $\Lambda^\bullet T^*_\C B$.\\[0.5em]
\indent-\textbf{\textit{Heat kernels.}} Assume the vertical subbundle $\ker T\pi$ is equipped with a metric which its restriction on each fibre defines Riemannian structures $g^{T_\R Z_b}$ on $ Z_b$. For the fibre-product bundle $M\times_\pi M=\{(x,y)\in M\times M\mid\pi(x)=\pi(y)\}$ with fibre at $b\in B$ equals to $Z_b\times Z_b$, we will also use the letter $\pi$ for the submersion $M\times_\pi M\rightarrow B$. In particular, $\pi^*\Lambda^\bullet T^*_\C B$ is a vector bundle over $M\times_\pi M$. Let $\pr_1$ resp. $\pr_2$ be the projection of $M\times_\pi M$ onto the first resp. second component. For two vector bundles $\mathcal{E}_1$ and $\mathcal{E}_2$ on $M$ set 
\[\mathcal{E}_1\boxtimes_{\pi}\mathcal{E}_2:=\pr_1^*\mathcal{E}_1\otimes\pr_2^*\mathcal{E}_2\]
which is a vector bundle over $M\times_\pi M$. Let $\mathscr{K}(\mathcal{E})$ be the bundle over $B$ whose fibre at $b\in B$ is the set of smoothing operators on $\Gamma(Z_b,\mathcal{E}_b)$ and let $K=\{K_b\}_{b\in B}\in\mathfrak{A}^\bullet (B,\mathscr{K}(\mathcal{E}))$ be a family of smoothing operators with differential form coefficients. Its kernel is an element of $\Gamma(M\times_\pi M,\pi^*\Lambda^\bullet T^*_\C B\otimes\mathcal{E}\boxtimes_\pi\mathcal{E}^*)$ and when restricted to the diagonal, it is a section of $\pi^*\Lambda^\bullet T^*_\C B\otimes\End(\mathcal{E})$ over $M$ where $M$ is identified with its embedding in $M\times_\pi M$ as the diagonal. The $\mathfrak{A}^\bullet(B)$\textit{-valued supertrace} $\Tr_s:\mathfrak{A}^\bullet(B,\mathscr{K}(\mathcal{E}))\rightarrow\mathfrak{A}^\bullet(B)$ sends $K$ to the differential form on $B$ given by
\[b\mapsto\int_{Z_b}\Tr_s\bigl\langle z\bigm| K_b\bigm|z\bigr\rangle\,\dvol_{g^{T_\R Z_b}}(z).\]
The integration along the fibres is denoted with $\int_Z$, thus, $\Tr_s[K]=\int_Z\Tr_s\bigl\langle z\bigm|K\bigm|z\bigr\rangle\,\dvol_{g^{T_\R Z}}(z)$. By \cite[Ch. 9.7 Appx.]{BeGeV04}, for $u>0$, there exists an unique heat kernel for the curvature $A^2$ of a superconnection adapted to $D$ with corresponding family of smoothing operators denoted by $\exp(-uA^2)$ or $e^{-uA^2}\in\mathfrak{A}^\bullet(B,\mathscr{K}(\mathcal{E}))$. Its construction shows that this also holds for elliptic superconnections. In particular, $e^{-uA^2}$ is fibrewise trace-class.\\[0.5em]
\indent-\textbf{\textit{Landau symbol.}} For a family of tensors $(\eta_u)_u$ on a manifold depending on $u\in\N_0$ or $u\in\R_{>0}$, we say that as $u\rightarrow\infty$ (resp. $u\searrow0$), $\eta_u=\eta+O(f(u))$ in the $C^\infty$ topology, if and only if for any compact set $K$ of the base manifold and any $m\in\N_0$, there exists $C_{m,K}>0$ such that for every $u\ge1$ (resp. $u\le1$), the norm of $\eta_u-\eta$ and all its derivatives of order $\le m$ over $K$ are bounded by $C_{m,K}f(u)$.

\section{Equivariant Holomorphic Torsion Forms}\label{sec 1}
In this section we give the definition of the equivariant holomorphic torsion forms of \cite[(4.11.3)]{B13}. Subsection \ref{subsec 1.1} summarizes the definitions of a Hermitian fibration and objects associated to it; Subsection \ref{subsec 1.2} is devoted to the Bismut superconnection and Subsection \ref{subsec 1.3} describes the construction of the torsion forms. We only recall the necessary tools from \cite{B13,P23} needed for the definition of the torsion forms; for a comprehensive treatment as well as a more geometric and adiabatic limit point of view on this subject, we refer to \cite{B13}.

\subsection{Hermitian Fibrations}\label{subsec 1.1}
Let $M$ and $B$ be complex manifolds and $\pi:M\rightarrow B$ be a holomorphic fibre bundle with compact fibre $Z$. Let $n,m$ and $d=n+m$ be the complex dimension of $Z,B$ and $M$. The holomorphic vertical subbundle of $T'M$ will be denoted by $TZ$ with fibres given by $T_pZ=(T'Z_{\pi(p)})_{|_p}$. The real vector bundle $T_\R Z$ is equipped with a complex structure $J^{T_\R Z}\in\Gamma(M,\End(T_\R Z))$. Let $\omega^M\in\mathfrak{A}^{1,1}(M)$ be a real $(1,1)$-form on $M$. We first recall the definition of a Hermitian fibration (\cite[p. 22]{B13}, \cite[Sec. 2.1]{P23}).
\begin{Definition}\label{Hermitian fibration}
The tuple $(\pi,\omega^M)$ is called  a \textit{Hermitian fibration} if for $X,Y\in\Gamma(M,T_\R Z)$,
\[g^{T_\R Z}(X,Y):=\omega^M(X,J^{T_\R Z}Y)\]
defines a metric\footnote[1]{Note that here the sign convention for $\omega^M$ differs from \cite{B13,P23}.} on $T_\R Z$. If additionally $\omega^M$ is closed, then $(\pi,\omega^M)$ is called a \textit{Kähler fibration}.
\end{Definition}
\noindent We will assume that the given data $(\pi,\omega^M)$ is in fact a Hermitian fibration. Let $T^HM\subset T'M$ be the orthogonal bundle to $TZ$ in $T'M$ with respect to $\omega^M$. We get a decomposition of smooth vector bundles
\[T'M= T^HM\oplus TZ,\qquad T_\R M= T^H_\R M\oplus T_\R Z\]
in horizontal and vertical parts. In general, $T^HM$ is not a holomorphic subbundle of $T'M$. One then has the isomorphism 
\[T^HM\cong\pi^*T'B,\qquad\Lambda^\bullet T^*_\R M\cong\pi^*\Lambda^\bullet T^*_\R B\otimes\Lambda^\bullet T^*_\R Z.\]
For $U\in\Gamma(B,T_\C B)$, we denote by $U^H\in\Gamma(M,T_\R^HM\otimes_\R\C)$ its horizontal lift. Set
\[\omega^H:=\omega^M_{|_{{T^H_\R M\otimes T^H_\R M}}},\qquad\omega^Z:=\omega^M_{|_{{T_\R Z\otimes T_\R Z}}}\]
and extend $\omega^H$ and $\omega^Z$ by zero to $(T_\R^HM\oplus T_\R Z)^{\otimes2}$ so that we have $\omega^M=\omega^H\oplus\omega^Z$.

Let $\pr_{T_\R Z}$ be the orthogonal projection onto the vertical subbundle. For an arbitrary Riemannian metric $g^{T_\R B}$ on $B$ and $g^{T_\R M}:=\pi^*g^{T_\R B}\oplus g^{T_\R Z}$, we define, if $X,Y,Z\in\Gamma(M,T_\R M),$
\begin{align*}
\nabla^{T_\R Z,LC}:=~&\pr_{T_\R Z}\circ\nabla^{T_\R M,LC},\\
\nabla^{T_\R M,\oplus}:=~&\pi^*\nabla^{T_\R B,LC}\oplus\nabla^{T_\R Z,LC},\\
T(X,Y):=~&\nabla^{T_\R M,\oplus}_XY-\nabla^{T_\R M,\oplus}_YX-[X,Y],\\
\rho(X,Y,Z):=~&g^{T_\R M}(\nabla^{T_\R M,LC}_XY-\nabla^{T_\R M,\oplus}_XY,Z).
\end{align*}
We cite some characteristics of these objects (\cite[Thm. 1.9]{B86}, \cite[Thm. 1.7]{BGS88b}): The connection $\nabla^{T_\R Z,LC}$ on $T_\R Z$ does not depend on the choice of $g^{T_\R B}$ and when restricted to a fibre, it is its Levi-Civita connection. The torsion $T$ of $\nabla^{T_\R M,\oplus}$ is an element of $\mathfrak{A}^2(M,T_\R Z)$, independent of $g^{T_\R B}$ and its complex linear extension lies in $\mathfrak{A}^{1,1}(M,T_\C Z)$. Finally, $\rho\in\Gamma(M,T_\R^*M^{\otimes 3})$ is also independent of $g^{T_\R B}$.

Let $\pi_\mathcal{E}:\mathcal{E}\rightarrow M$ be a holomorphic vector bundle over $M$. As mentioned in Section \ref{sec 0}, we can associate to the family of vector bundles $\{\Lambda^{0,\bullet}T^*Z_b\otimes\mathcal{E}_b\}_{b\in B}$ the infinite dimensional vector bundle $\{\Gamma(Z_b,(\Lambda^{0,\bullet}T^*Z_b\otimes\mathcal{E}_b)\}_{b\in B}=\{\mathfrak{A}^{0,\bullet}(Z_b,\mathcal{E}_b)\}_{b\in B}$ over $B$ with the identification
\[\mathfrak{A}^\bullet(B,\{\mathfrak{A}^{0,\bullet}(Z_b,\mathcal{E}_b)\}_{b\in B})=\Gamma(M,\pi^*\Lambda^\bullet T^*_\C B\otimes\Lambda^{0,\bullet}T^*Z\otimes\mathcal{E}).\]
Let $h^{TZ}$ be the Hermitian metric on $TZ$ obtained from $g^{T_\R Z}$ and let $h^{T^{1,0}Z}$ be the Hermitian metric on $T^{1,0}Z$ induced by $h^{TZ}$ via the isomorphism $(T_\R Z,J^{T_\R Z})\cong T^{1,0}Z,~X\mapsto\sqrt{2}X^{1,0}$. Let $h^\mathcal{E}$ be a Hermitian metric on $\mathcal{E}$ and $h^{\Lambda^{0,\bullet}T^*Z\otimes\mathcal{E}}$ the Hermitian metric on $\Lambda^{0,\bullet}T^*Z\otimes\mathcal{E}$ induced by $h^{T^{1,0}Z},h^\mathcal{E}$. We define a Hermitian product $\langle\cdot,\cdot\rangle_{L^2}$ on $\{\mathfrak{A}^{0,\bullet}(Z_b,\mathcal{E}_b)\}_{b\in B}$ associated to $h^{TZ}$ and $h^\mathcal{E}$ by, for $s_1,s_2\in\Gamma(B,\{\mathfrak{A}^{0,\bullet}(Z_b,\mathcal{E}_b)\}_{b\in B})$,
\[\langle s_1,s_2\rangle_{L^2,b}:=\frac{1}{(2\pi)^{\dim_\C Z}}\int_{Z_b}h^{\Lambda^{0,\bullet}T^*Z\otimes\mathcal{E}}(s_1,s_2)\,\dvol_{g^{T_\R Z_b}}.\]

We will use the following notations regarding local bases: Let $\{w_i\}_i$ be an local orthonormal basis of $(T^{1,0}Z,h^{T^{1,0}Z})$. An orthonormal basis $\{e_i\}_i$ of $(T_\R Z, g^{T_\R Z})$ is given by $e_{2i-1}=\frac{1}{\sqrt{2}}(w_i+\overline{w}_i)$ and $e_{2i}=\frac{1}{\sqrt{2}}(w_i-\overline{w}_i)$. Let $\{f_\alpha\}_\alpha$ be a basis of $T_\R B$ which will be also identified as a basis of $T^H_\R M$. Latin indices $i,j,\ldots$ will label vertical variables and greek indices $\alpha,\beta,\ldots$ label horizontals variables. If $a_j$ is greek or latin with corresponding $\mathfrak{e}_{a_j}\in\{e_i\}_i\cup\{f_\alpha\}_\alpha$, then for $R\in T_\R^*M^{\otimes q}$, we write $R_{a_1,\ldots,a_q}:=R(\mathfrak{e}_{a_1},\ldots,\mathfrak{e}_{a_q})$. Dual bases are denoted by upper indices. These notations will be used throughout this paper.

\subsection{Bismut Superconnection}\label{subsec 1.2}
Let  $\Cl(T_\R Z,g^{T_\R Z})=\bigoplus_j T_\R Z^{\otimes j}/I$ be the Clifford algebra\footnote[2]{In \cite[Ch. 3.1]{B13}, the Clifford algebra with respect to $g^{T_\R Z}/2$ was considered which is also common in complex geometry. Here, we follow \cite[p. 88]{P23}.} of $(T_\R Z, g^{T_\R Z})$ where $I$ is the ideal in the tensor algebra generated by  $\{v\otimes v+g^{T_\R Z}(v,v)\mid v\in T_\R Z\}$. Then $\Cl(T_\R Z,g^{T_\R Z})\otimes_\R\C=\Cl(T_\C Z,g^{T_\C Z})$ with $g^{T_\C Z}$ the complex bilinear extension of $g^{T_\R Z}$. For $X\in\Gamma(M,T_\C Z)$ let 
\[X^{1,0\flat}:=g^{T_\C Z}(X^{1,0},\cdot)\in\Gamma(M,(T^{0,1}Z)^*)\]
be the metric dual of $X^{1,0}$. The bundle $\Lambda^{0,\bullet}T^*Z\otimes\mathcal{E}$ becomes a $\Cl(T_\C Z,g^{T_\C Z})$-module with Clifford action induced by, if $\alpha\in\Gamma(M,\Lambda^{0,\bullet}T^*Z)$ and $s\in\Gamma(M,\mathcal{E})$,
\[c(X)(\alpha\otimes s):=\sqrt{2}(X^{1,0\flat}\wedge\alpha-\iota_{X^{0,1}}\alpha)\otimes s.\]
The inclusion $T_\C Z\overset{c}{\xhookrightarrow{}}\Cl(T_\C Z,g^{T_\C Z})$ extends to $\pi^*\Lambda^\bullet T^*_\C B\otimes T_\C Z\overset{c}{\xhookrightarrow{}}\pi^*\Lambda^\bullet T^*_\C B\otimes\Cl(T_\C Z,g^{T_\C Z})$ as $\id\otimes c$.
Let  $\Cl(Z_b,g^{T_\R Z_b})$ be the Clifford algebra of $(Z_b,g^{T_\R Z_b})$ and set $T^H:=T_{|_{{T^H_\R M\otimes T^H_\R M}}}$ which will be viewed as a section of $\pi^*\Lambda^2T^*_\R B\otimes T_\R Z$. Then we have
\[c(T^H)\in\Gamma(M,\pi^*\Lambda^2T^*_\R B\otimes\Cl(T_\R Z,g^{T_\R Z}))=\Gamma(B,\Lambda^2T^*_\R B\otimes\{\Cl(Z_b,g^{T_\R Z_b}\}_{b\in B}).\] 
Define the quantization map $(\cdot)^c:\Lambda^\bullet T^*_\R Z\rightarrow\Cl(T_\R Z,g^{T_\R Z})$ given by
\[(e^{i_1}\wedge\ldots\wedge e^{i_k})^c:=\frac{1}{2^{k/2}}c(e_{i_1})\ldots c(e_{i_k})\quad\text{for}\quad 1\le i_1<\ldots<i_k\le 2n\]
and extend it to $\Lambda^\bullet T^*_\R M\cong\pi^*\Lambda^\bullet T^*_\R B\otimes\Lambda^\bullet T^*_\R Z\overset{(\cdot)^c}{\rightarrow}\pi^*\Lambda^\bullet T^*_\R B\otimes\Cl(T_\R Z,g^{T_\R Z})$ as well as $\C$-linearly to $\Lambda^\bullet T^*_\C M$.

Let $\nabla^{TZ},\nabla^\mathcal{E}$ be the Chern connections of $(TZ,h^{TZ}),(\mathcal{E},h^\mathcal{E})$. Assume temporarily that $T_\R Z$ is equipped with a spin structure which is equivalent to $\det(TZ)$ having a square root $\lambda$ (\cite[Thm. 2.2]{H74}). The line bundle $\lambda$ is then a Hermitian holomorphic vector bundle on $M$ with Chern connection $\nabla^\lambda$. Let 
\[\mathcal{S}^{TZ}:=\Lambda^{0,\bullet}T^*Z\otimes\lambda^{-1}\]
be the spinor bundle of $(T_\R Z,g^{T_\R Z})$ with connection $\nabla^{\mathcal{S}^{TZ,LC}}$ induced by $\nabla^{T_\R Z,LC}.$ Then $\Lambda^{0,\bullet}T^*Z\otimes\mathcal{E}=\mathcal{S}^{TZ}\otimes\lambda\otimes\mathcal{E}$ has a connection $\nabla^{\Lambda^{0,\bullet}T^*Z\otimes\mathcal{E},LC}$ induced by $\nabla^{\mathcal{S}^{TZ,LC}},\nabla^\lambda,\nabla^{\mathcal{E}}$. Because the square root of $\det(TZ)$ always exists locally, the connection $\nabla^{\Lambda^{0,\bullet}T^*Z\otimes\mathcal{E},LC}$ is always defined. 

Define $k\in\Gamma(B,T^*_\R B\otimes\{C^\infty(Z_b)\}_{b\in B})$ by, if $U\in\Gamma(B,T_\R B)$, 
\[k(U):=\sum_{i=1}^{2n}e^i(T(U^H,e_i))=\Tr^{T_\R Z}\bigl[T(U^H,\cdot)_{|_{T_\R Z}}\bigr]\in\Gamma(B,\{C^\infty(Z_b)\}_{b\in B})=C^\infty(M).\]
For $U\in\Gamma(B,T_\R B)$ and $s\in\Gamma(B,\{\mathfrak{A}^{0,\bullet}(Z_b,\mathcal{E}_b)\}_{b\in B})$, set
\[\nabla^{\{\mathfrak{A}^{0,\bullet}(Z_b,\mathcal{E}_b)\}_{b\in B}}_Us:=\nabla^{\Lambda^{0,\bullet}T^*Z\otimes\mathcal{E},LC}_{U^H}s+\frac{1}{2}k(U).\]
This connection on $\{\mathfrak{A}^{0,\bullet}(Z_b,\mathcal{E}_b)\}_{b\in B}$ extends to $\Lambda^\bullet T_\C^*B\otimes\{\mathfrak{A}^{0,\bullet}(Z_b,\mathcal{E}_b)\}_{b\in B}$ by $\C$-linearly in $U$, the de Rham operator $d^B$ of $B$ and Leibniz's rule (cf. Section \ref{sec 0}).

For $u>0$, let $\psi_u$ be the automorphism of $\Lambda^\bullet T_\R^* B$ such that if $\alpha\in\Lambda^kT_\R^*B$, then 
\[\psi_u\alpha:=u^k\alpha.\]
\begin{Definition}[cf. {\cite[Thm. 3.8.1]{B13}}] The \textit{Levi-Civita superconnection} $A^{LC}$, the \textit{Bismut superconnection} $B$ and its rescaled version $B_u$ for $u>0$ on $\{\mathfrak{A}^{0,\bullet}(Z_b,\mathcal{E}_b)\}_{b\in B}$ are defined by
\begin{align*}
&A^{LC}:=\nabla^{\{\mathfrak{A}^{0,\bullet}(Z_b,\mathcal{E}_b)\}_{b\in B}}+\sum_{i=1}^{2n}\frac{c(e_i)}{\sqrt{2}}\nabla^{\Lambda^{0,\bullet}T^*Z\otimes\mathcal{E},LC}_{e_i}-\frac{c(T^{H})}{2\sqrt{2}},\\
&B:=A^{LC}-\frac{1}{2}\left((\overline{\partial}^M-\partial^M)i\omega^M\right)^c,\\
&B_u:=\sqrt{u}\psi_{1/\sqrt{u}}B\psi_{\sqrt{u}}.
\end{align*}
\end{Definition}
For $b\in B$, let $\overline{\partial}^{Z_b}$ be the Dolbeault operator acting on $\Gamma(Z_b,\Lambda^{0,\bullet}T^*Z_b\otimes\mathcal{E}_b)$. Let $\overline{\partial}^{Z_b*}$ be its formal adjoint with respect to $\langle\cdot,\cdot\rangle_{L^2,b}$ and put $\overline{\partial}^Z+\overline{\partial}^{Z*}:=\{\overline{\partial}^{Z_b}+\overline{\partial}^{Z_b*}\}_{b\in B}$. By \cite[Thm. 3.7.3]{B13}, $B^{(0)}=\overline{\partial}^Z+\overline{\partial}^{Z*}$. In particular, $B$ is elliptic. Furthermore, $B^2_u=u\psi_{1/\sqrt{u}}B^2\psi_{\sqrt{u}}$ and $e^{-B^2_u}=\psi_{1/\sqrt{u}}\bigl(e^{-{uB^2}}\bigr)$.

In the case of a Kähler fibration, these objects simplify as following: Let $\nabla^{\Lambda^{0,\bullet}T^*Z\otimes\mathcal{E}}$ be the connection on $\Lambda^{0,\bullet}T^*Z\otimes\mathcal{E}$ induced by $\nabla^{TZ},\nabla^\mathcal{E}$. Then
\begin{align*}
&\nabla^{\Lambda^{0,\bullet}T^*Z\otimes\mathcal{E},LC}=\nabla^{\Lambda^{0,\bullet}T^*Z\otimes\mathcal{E}},\quad k=0,\quad \sum_{i=1}^{2n}\frac{c(e_i)}{\sqrt{2}}\nabla^{\Lambda^{0,\bullet}T^*Z\otimes\mathcal{E},LC}_{e_i}=\overline{\partial}^Z+\overline{\partial}^{Z*},\quad B=A^{LC}.
\end{align*}
For $u>0$, define a fibrewise connection $\overline{\nabla}_u$ on $\pi^*\Lambda^\bullet T^*_\R B\otimes\Lambda^{0,\bullet}T^*Z\otimes\mathcal{E}$ by
\begin{align*}
\overline{\nabla}_{u,e_i}:=~&\nabla^{\Lambda^{0,\bullet}T^*Z\otimes\mathcal{E},LC}_{e_i}+\frac{1}{\sqrt{2u}}\sum_{j,\alpha}\rho_{i,j,\alpha}c(e_j)f^\alpha+\frac{1}{2u}\sum_{\alpha,\beta}\rho_{i,\alpha,\beta}f^\alpha f^\beta\\
&-\frac{1}{2}\psi_{1/\sqrt{u}}\left(\iota_{e_i}(\overline{\partial}^M-\partial^M)i\omega^M\right)^c\psi_{\sqrt{u}}.
\end{align*}
We will use the following standard notation: if $C\in\Gamma(M,T_\R^*Z\otimes\pi^*\Lambda^\bullet T^*_\R B\otimes\End(\Lambda^{0,\bullet}T^*Z\otimes\mathcal{E}))$, then
\begin{align*}
\sum\left(\nabla^{\Lambda^{0,\bullet}T^*Z\otimes\mathcal{E},LC}_{e_i}+C(e_i)\right)^2:=&\sum_{i=1}^{2n}\left(\nabla^{\Lambda^{0,\bullet}T^*Z\otimes\mathcal{E},LC}_{e_i}+C(e_i)\right)^2\\
&-\nabla^{\Lambda^{0,\bullet}T^*Z\otimes\mathcal{E},LC}_{\sum_{i=1}^{2n}\nabla^{T_\R Z,LC}_{e_i}e_i}-C\Bigl(\sum_{i=1}^{2n}\nabla^{T_\R Z,LC}_{e_i}e_i\Bigr).
\end{align*}
Let $s^Z$ be the scalar curvature of $(Z,g^{T_\R Z})$. By the Bismut-Lichnerowicz formula (\cite[Thm. 3.9.3]{B13}, see also \cite[Thm. 2.7]{P23}), one has the identity
\begin{align}\label{Lichnerowicz}
B_u^2=&-\frac{u}{2}\sum(\overline{\nabla}_{u,e_i})^2+u\frac{s^Z}{8}+\frac{u}{4}\sum_{i,j}c(e_i)c(e_j)\Bigl(\Omega^\mathcal{E}+\frac{1}{2}\Tr\Omega^{TZ}\Bigr)(e_i,e_j)\nonumber\\
&+\sqrt{\frac{u}{2}}\sum_{i,\alpha}c({e}_i)f^\alpha\Big(\Omega^\mathcal{E}+\frac{1}{2}\Tr\Omega^{TZ}\Big)(e_i,f_\alpha)+\sum_{\alpha,\beta}\frac{f^\alpha f^\beta}{2}\Bigl(\Omega^\mathcal{E}+\frac{1}{2}\Tr\Omega^{TZ}\Bigr)(f_\alpha,f_\beta)\nonumber\\
&+u\psi_{1/\sqrt{u}}\left(\overline{\partial}^M\partial^Mi\omega^M\right)^c+\frac{u}{16}\left\|(\overline{\partial}^Z-\partial^Z)\omega^Z\right\|_{\Lambda^\bullet T^*_\R Z}.
\end{align}

\subsection{Equivariant Holomorphic Torsion Forms}\label{subsec 1.3}
\begin{Definition}\label{equiv struc} Let $\gamma$ be a holomorphic isometry of $(M,\omega^M)$ which preserves the fibres, i.e. $\pi\circ\gamma=\pi$. A $\gamma$-\textit{equivariant structure} on $(\mathcal{E},h^\mathcal{E})$ is a holomorphic map $\gamma^\mathcal{E}:\mathcal{E}\rightarrow\mathcal{E}$ such that $\pi\circ\gamma^\mathcal{E}=\gamma\circ\pi^\mathcal{E}$, $\gamma^\mathcal{E}:\mathcal{E}_x\rightarrow\mathcal{E}_{\gamma(x)}$ is linear for all $x\in M$ and $h^\mathcal{E}$ is preserved by $\gamma^\mathcal{E}$.
\end{Definition}
\noindent We assume that a $\gamma$-equivariant structure is given. One has an induced action on $\Gamma(M,\mathcal{E})$ by $(\gamma s)(x):=\gamma^\mathcal{E}s(\gamma^{-1}x)$. In general, $\gamma,\gamma^\mathcal{E}$ induce actions on various tensor bundles and sections which will often, if not stated otherwise, simply be denoted by $\gamma$ again. In particular, $\gamma$ acts on $\mathfrak{A}^\bullet(B,\{\mathfrak{A}^{0,\bullet}(Z_b,\mathcal{E}_b)\}_{b\in B})$. Let $M^\gamma\subset M$ and $Z^\gamma\subset Z$ be the fixed-point sets of $\gamma$. 

We make the assumption that the direct image $R^\bullet\pi_*\mathcal{E}$ of $\mathcal{E}$ by $\pi$ is locally free. For each base point $b\in B$, let $H^\bullet(Z_b,\mathcal{E}_b)$ be the cohomology of the sheaf of holomorphic sections of $\mathcal{E}$ over the fibre $Z_b$. Then by the assumption, the $H^\bullet(Z_b,\mathcal{E}_b)$ form a $\Z$-graded holomorphic vector bundle denoted by $H^\bullet(Z,\mathcal{E}_{|_Z})$ on $B$ and $R^\bullet\pi_*\mathcal{E}=H^\bullet(Z,\mathcal{E}_{|_Z})$. By Hodge theory, for every $b\in B$,
\begin{align*}
H^\bullet(Z_b,\mathcal{E}_b)\cong\ker(\overline{\partial}^{Z_b}+\overline{\partial}^{Z_b*}).
\end{align*}
From this isomorphism and the restriction of the Hermitian product $\langle\cdot,\cdot\rangle_{L^2}$ onto the right side, $h^{TZ}$ and $h^\mathcal{E}$ induce a $\gamma$-invariant metric $h^{H^\bullet(Z,\mathcal{E}_{|_Z})}$ on $H^\bullet(Z,\mathcal{E}_{|_Z})$ for which the $H^k(Z,\mathcal{E}_{|_Z})$ are mutually orthogonal. Let $\nabla^{H^\bullet(Z,\mathcal{E}_{|_Z})}$ be the Chern connection of $(H^\bullet(Z,\mathcal{E}_{|_Z}),h^{H^\bullet(Z,\mathcal{E}_{|_Z})})$.

Let $\omega^{H\overline{H}}\in\Gamma(B,\Lambda^2T^*_\R B\otimes\{C^\infty(Z_b)\}_{b\in B})$ be given by, if $U,V\in\Gamma(B,T_\R B)$,
\[\omega^{H\overline{H}}(U,V):=\omega^M(U^H,V^H).\]
\begin{Definition} Let $N_V$ be the operator on $\Lambda^{0,\bullet}T^*Z$ acting by multiplication with $k$ on $\Lambda^{0,k}T^*Z$. Then $N_V$ acts on $\mathfrak{A}^\bullet(B,\{\mathfrak{A}^{0,\bullet}(Z_b,\mathcal{E}_b)\}_{b\in B})$. For $u>0$, the \textit{number operator} $N_u$ on $\mathfrak{A}^\bullet(B,\{\mathfrak{A}^{0,\bullet}(Z_b,\mathcal{E}_b)\}_{b\in B})$ is given by
\[N_u:=N_V-i\frac{\omega^{H\overline{H}}}{u}.\]
\end{Definition}
Let $\vp\in\Gamma(B,\End(\Lambda^\bullet T^*_\C B))$ be the endomorphism defined by, if $\alpha$ is a homogeneous differential form,
\[\vp:\alpha\mapsto(2\pi i)^{-\nicefrac{\mathrm{deg}\alpha}{2}}\alpha.\]
\begin{Theorem}[see {\cite[Prop. 4.6.1, Thm. 4.10.4]{B13}}]\label{asymp zeta} The following asymptotics holds in the $C^\infty$ topology on $B$:
\begin{enumerate}[label=\alph*)]
\item There exist locally computable forms $c_{\gamma,j},C_{\gamma,j}\in\bigoplus\limits_{p\ge0}\mathfrak{A}^{p,p}(B),~j\ge d,$ such that for any $k\in\N_0$, as $u\searrow0,$\\
\mbox{}\hfil$\begin{aligned}[t]
\vp\Tr_s\bigl[\gamma\exp(-B_u^2)\bigr]&=\sum_{j=-d}^kc_{\gamma,j}u^j+O(u^{k+1}),\\
\vp\Tr_s\bigl[\gamma N_u\exp(-B_u^2)\bigr]&=\sum_{j=-d}^kC_{\gamma,j}u^j+O(u^{k+1}).
\end{aligned}$
\item As $u\rightarrow\infty,$\\ 
\mbox{}\hfil$\begin{aligned}[t]
&\vp\Tr_s\bigl[\gamma N_u\exp(-B_u^2)\bigr]=\vp\Tr_s^{H^\bullet(Z,\mathcal{E}_{|_Z})}\bigl[\gamma N_Ve^{-\Omega^{H^\bullet(Z,\mathcal{E}_{|_Z})}}\bigr]+O\Bigl(\frac{1}{\sqrt{u}}\Bigr).
\end{aligned}$
\end{enumerate}
\end{Theorem}
\noindent For $s\in \C$, define formally the zeta function 
\[\zeta_\gamma(s):=-\frac{1}{\Gamma(s)}\int_0^\infty u^{s-1}\vp\left(\Tr_s\bigl[\gamma N_u\exp(-B_u^2)\bigr]-\Tr_s^{H^\bullet(Z,\mathcal{E}_{|_Z})}\bigl[\gamma N_Ve^{-\Omega^{H^\bullet(Z,\mathcal{E}_{|_Z})}}\bigr]\right)\,du.\]
By splitting the integral into $\int_0^\infty=\int_0^1+\int_1^\infty$, one can separate $\zeta_\gamma=\zeta^{[0,1]}_\gamma+\zeta^{[1,\infty[}_\gamma$. Due to Theorem \ref{asymp zeta}, $\zeta^{[0,1]}_\gamma$ resp. $\zeta^{[1,\infty[}_\gamma$ are well-defined for $\Re s>d$ resp. $\Re s<\frac{1}{2}$ and $\zeta^{[0,1]}_\gamma$ extends to a holomorphic function near $s=0$. Thus, $\zeta_\gamma$ is a well-defined holomorphic function in $s$ near $0$.
\begin{Definition}\label{holom Torsion} The \textit{equivariant holomorphic torsion form} is defined as
\[T_\gamma(\omega^M,h^\mathcal{E}):=\zeta'_\gamma(0).\]
The components in the different degrees of $T_\gamma(\omega^M,h^\mathcal{E})$ are referred as \textit{equivariant holomorphic torsion forms}.
\end{Definition}
From Theorem \ref{asymp zeta}, one derives the equality 
\begin{align*}
T_\gamma(\omega^M,h^\mathcal{E})=&-\int_0^1\biggl(\vp\Tr_s\bigl[\gamma N_u\exp(-B_u^2)\bigr]-\sum_{j=-d}^0 C_{\gamma,j}u^j\biggr)\,\frac{du}{u}\\
&-\int_0^1\vp\biggl(\Tr_s\bigl[\gamma N_u\exp(-B_u^2)\bigr]-\Tr_s^{H^\bullet(Z,\mathcal{E}_{|_Z})}\bigr[\gamma N_Ve^{-\Omega^{H^\bullet(Z,\mathcal{E}_{|_Z})}}\bigr]\biggr)\,\frac{du}{u}\\
&+\sum_{j=-d}^{-1}\frac{C_{\gamma,j}}{j}+\Gamma'(1)\Bigl(C_{\gamma,0}-\vp\Tr_s^{H^\bullet(Z,\mathcal{E}_{|_Z})}\bigl[\gamma N_Ve^{-\Omega^{H^\bullet(Z,\mathcal{E}_{|_Z})}}\bigr]\Bigr).
\end{align*}
Moreover, by \cite[Thm. 4.11.2]{B13}, the torsion form lies in $\bigoplus\limits_{p\ge0}\mathfrak{A}^{p,p}(B)$ and satisfies the equation
\[\frac{\overline{\partial}\partial}{2\pi i}T_\gamma(\omega^M,h^\mathcal{E})=\ch_\gamma\bigl(H^\bullet(Z,\mathcal{E}_{|_Z}),h^{H^\bullet(Z,\mathcal{E}_{|_Z})}\bigr)-c_{\gamma,0}.\]
In the case of a Kähler fibration, by \cite[Thm. 2.10, Thm. 2.12]{M00}, $C_{\gamma,j}=0$ for $j<-1$, $c_{\gamma,j}=0$ for $j<0$ and one has the explicit identity
\[c_{\gamma,0}=\int_{Z^\gamma}\Td_\gamma(TZ,h^{TZ})\ch_\gamma(\mathcal{E},h^\mathcal{E}).\]
Here, $\ch_\gamma,\Td_\gamma$ are given as follows: For a $(q,q)$ matrix $A$, put
\[\Td(A):=\det\Bigl(\frac{A}{1-e^{-A}}\Bigr),\qquad\ctop(A):=\det A.\]
The genera associated to $\Td,\ctop$ are the Todd genus and the top Chern class. Let $(E,h^E)$ be an arbitrary Hermitian holomorphic vector bundle over a complex manifold $Z$ with a $\gamma$-equivariant structure and Chern connection $\nabla^E$. The distinct eigenvalues $1,e^{i\theta_1},\ldots,e^{i\theta_q} ~~(0<\theta_j <2\pi)$ of $\gamma$ are locally constant on $Z^\gamma$. Let $\{E^{\theta_j}\}_{0\le j\le q}$ with $\theta_0:=0$ be the corresponding eigenbundles with Hermitian metrics $h^{E^{\theta_j}}$ induced by $h^E$ and Chern connections $\nabla^{E^{\theta_j}}$. Set
\begin{align*}
&\ch_\gamma(E,h^E):=\Tr^E\Bigl[\gamma\exp\Bigl(\frac{-\Omega^E_{|_{Z^\gamma}}}{2\pi i}\Bigr)\Bigr],\\
&\Td_\gamma(E,h^E):=\Td \Bigl(\frac{-\Omega^{E^0}_{|_{Z^\gamma}}}{2\pi i}\Bigr)\prod_{j=1}^q\Bigl(\frac{\Td}{\ctop}\Bigr)\Bigl(\frac{-\Omega^{E^{\theta_j}}_{|_{{Z^\gamma}}}}{2\pi i}+i\theta_j\Bigr).
\end{align*}
These are closed differential forms on $Z^\gamma$ and their cohomology class does not depend on the metric. 
\section{Asymptotic Equivariant Torsion Forms}\label{sec 2}
Let $(\mathcal{L},h^\mathcal{L})$ be a Hermitian holomorphic line bundle on $M$ with a $\gamma$-equivariant structure. Let $\nabla^\mathcal{L}$ be the Chern connection of $(\mathcal{L},h^\mathcal{L})$. We assume that the $(1,1)$-form $i\Omega^\mathcal{L}$ is fibrewise positive, that is, for any $X\in T^{1,0}Z\setminus\{0\}$,
\begin{align}\label{pos}
\Omega^\mathcal{L}(X,\overline{X})>0.
\end{align}
For $p\in\N_0$, put $\mathcal{L}^p:=\mathcal{L}^{\otimes p}$. Now all the construction from Section \ref{sec 1} will be repeated here for $(\mathcal{E}\otimes\mathcal{L}^p,h^{\mathcal{E}\otimes\mathcal{L}^p})$ instead of $(\mathcal{E},h^\mathcal{E})$. The corresponding objects in this situation will then have the following adjusted notations:
\begin{align*}
B_{p,u}~\widehat{=}~&\text{rescaled Bismut superconnection},\quad B_p:=B_{p,1},\quad\overline{\partial}_p+\overline{\partial}_p^*:=B_p^{(0)},\\
\zeta_{\gamma,p}~\widehat{=}~&\text{corresponding zeta function.} 
\end{align*}
\begin{Remark}
In the preprint version \cite{P15}, the situation of only one fibration $M\rightarrow B$ like in this paper has been explored in great detail in \cite[Sec. 2]{P15}. That section did not appear in the final version \cite{P23} since it has been generalized and would be redundant (cf. \cite[Remark 1.4]{P23}). Thus, because our situation is simpler and we do not need the theory of Toeplitz operators, we will occasionally cite from \cite{P15} for more transparency.
\end{Remark}

The aim of this section is to prove Theorem \ref{Th1}. Subsection \ref{subsec 2.1} explains the technique of localization near the fixed-point manifold and replacing the manifold with the tangent space; Subsection \ref{subsec 2.2} provides the asymptotic expansion of the trace of the kernel for $u\in]0,1]$, whereas Subsection \ref{subsec 2.3} deals with the analysis for $u\ge1$.

\subsection{Localization near the Fixed-point Set}\label{subsec 2.1}
The localization technique in this subsection will be as in \cite[Sec. 3.4]{P23} (see also \cite[Ch. 1.6.2]{MMa07}), which was developed by Bismut-Lebeau (\cite[Sec. XIII b)]{BL91}).

Let $b_0\in B$ be a fixed base point. Because in this subsection we will work along $Z_{b_0}$, this fibre will be simply denoted by $Z$. In particular, $g^{T_\R Z_{b_0}}$ will be denoted by $g^{T_\R Z}$ as well as the corresponding Riemannian volume form. For the basis of $T_{\R,b_0}B$, we will also write $f_\alpha$ instead of $f_{\alpha,b_0}$.

By \cite[Thm. 1.5.8]{MMa07}, the operator $(\overline{\partial}_p+\overline{\partial}_p^*)^2$ has a spectral gap property, i.e. there exist a constant $C_\mathcal{L}>0$ depending on $\mathcal{L}$ and $\mu>0$ such that for its spectrum,
\begin{align*}
\Sp(\overline{\partial}_p+\overline{\partial}_p^*)^2\subset\{0\}\cup]p\mu-C_\mathcal{L},\infty[.
\end{align*}
By \cite[(3.61)]{P23} (see also \cite[Prop. 9.2]{B97}), one has $\Sp B^2_p=\Sp(\overline{\partial}_p+\overline{\partial}_p^*)^2$; therefore, $B^2_p$ has this property as well which will be needed in Subsection \ref{subsec 2.3}. Note that the spectral gap does not hold if the line bundle is only required to be semipositive (\cite{Do03}). Let $\Gamma$ be the contour indicated in the following figure:\\[1cm]
\begin{figure}[H]
\begin{center}
\begin{tikzpicture}
\draw[->] (-4,0) -- (3,0);
\draw[->] (-2,-1.8) -- (-2,1.8);
\draw[line width=1.2,-](-3,-1) -- (-3,1);	
\draw[line width=1.2,-](-3,-1) -- (3,-1);	
\draw[line width=1.2,-](-3,1) -- (3,1);	
\draw[,-{Latex[scale=1.2]}] (-2.5,1) -- (-2.6,1);
\draw[,-{Latex[scale=1.2]}] (-1.6,-1) -- (0,-1);
\filldraw (-2,0) circle (1pt);
\filldraw (-2,1) circle (1pt);
\filldraw (-2,-1) circle (1pt);
\filldraw (-3,0) circle (1pt);
\draw[] (3.3,0) node {$\R$};
\draw[] (-2,2.1) node {$i\R$};
\draw[] (-1.8,-0.25) node {$0$};
\draw[] (-2.2,1.25) node {$i$};
\draw[] (-2.35,-1.25) node {$-i$};
\draw[] (1.5,-0.7) node {$\Gamma$};
\draw[] (-3.5,-0.28) node {$-C_\mathcal{L}$};
\end{tikzpicture}
\end{center}
\caption{Contour $\Gamma$.}
\end{figure}
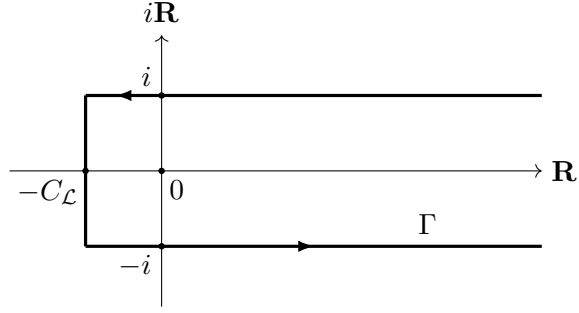
\noindent The resolvent $(\lambda-B_p^2)^{-1}$ exists for $\lambda\in\Gamma$ and by \cite[(3.65)]{P23}, there are $k,k'\in\N_0$ such that $\|(\lambda-B_p^2)^{-1}\|_\infty\le|\lambda|^kp^{k'}$, where $\|\cdot\|_\infty$ denotes the operator norm. For $\vr>0$, let $f_\vr:\R\rightarrow[0,1]$ be a smooth even function with $f_\vr(t)=\begin{cases}1~,~|t|<\frac{\vr}{2},\\0~,~|t|>\vr\end{cases}$. For $a\in\C$ and $u>0$, define the functions
\begin{align*}
F_{u,\vr}(a)&:=\frac{1}{\sqrt{2\pi}}\int_{-\infty}^\infty e^{is\sqrt{2}a}\exp(-s^2/2)f_\vr(\sqrt{u}s)\,ds,\\
G_{u,\vr}(a)&:=\frac{1}{\sqrt{2\pi}}\int_{-\infty}^\infty e^{is\sqrt{2}a}\exp(-s^2/2)(1-f_\vr(\sqrt{u}s))\,ds
\end{align*}
which are even holomorphic functions. Hence, there exist holomorphic functions $\tilde{F}_{u,\vr},\tilde{G}_{u,\vr}$ with $\tilde{F}_{u,\vr}(a^2)=F_{u,\vr}(a)$ and $\tilde{G}_{u,\vr}(a^2)=G_{u,\vr}(a)$. The restriction of $\tilde{F}_{u,\vr},\tilde{G}_{u,\vr}$ on $\R$ lie in the Schwartz space $\mathscr{S}(\R)$. Consequently, $\tilde{F}_{u,\vr}(B_p^2)=\int_\Gamma\tilde{F}_{u,\vr}(\lambda)(\lambda-B_p^2)^{-1}\,d\lambda$ is well-defined as well as $\tilde{G}_{u,\vr}(B_p^2)$ and for $t>0$ one has
\begin{align}\label{F+G=B}
\exp(-tB_p^2)=\tilde{F}_{u,\vr}(tB_p^2)+\tilde{G}_{u,\vr}(tB_p^2).
\end{align}
By the Schwartz kernel theorem, $\exp(-tB_p^2),\tilde{G}_{u,\vr}(tB_p^2)$ and $\tilde{F}_{u,\vr}(tB_p^2)$ are represented by smooth kernels $\bigl\langle z\bigm|\exp(-tB_p^2)\bigm|z'\bigr\rangle,\bigl\langle z\bigm|\tilde{G}_{u,\vr}(tB_p^2)\bigm|z'\bigr\rangle$ and $\bigl\langle z\bigm|\tilde{F}_{u,\vr}(tB_p^2)\bigm|z'\bigr\rangle$ with respect to $\dvol_{g^{T_\R Z}}(z')$. Define the vector bundle $\mathbf{E}_p$ over $Z$ by
\[\mathbf{E}_p:=\Lambda^\bullet T^*_{\R,b_0}B\otimes\Lambda^{0,\bullet}T^*Z\otimes\mathcal{E}\otimes\mathcal{L}^p\] 
where $\Lambda^\bullet T^*_{\R,b_0}B$ is a trivial bundle. $\mathbf{E}_p$ will be equipped with the Hermitian metric $h^{\mathbf{E}_p}$ induced by $g^{T_\R B},h^{\Lambda^{0,\bullet}T^*Z\otimes\mathcal{E}},h^\mathcal{L}$ and with connection $\nabla^{\mathbf{E}_p}$ induced by $\nabla^{T_\R B,LC},\nabla^{\Lambda^{0,\bullet}T^*Z\otimes\mathcal{E},LC},\nabla^\mathcal{L}$. Let $h^{\mathbf{E}_p\boxtimes\mathbf{E}^*_p}$ and $\nabla^{\mathbf{E}_p\boxtimes\mathbf{E}^*_p}$ be the induced metric and connection on $\mathbf{E}_p\boxtimes\mathbf{E}^*_p$. In \cite[Prop. 3.8]{P23}, it was shown that for any $m\in\N_0$ there exist $C>0$ and $N\in\N$ such that for any $u>0,p\in\N$,
\begin{align}\label{G Abschätzung}
\Bigl\|\bigl\langle\cdot\bigm|\tilde{G}_{\frac{u}{p},\vr}\Bigl(\frac{u}{p}B^2_p\Bigr)\bigm|\cdot\bigr\rangle\Bigr\|_{C^m(M\times_\pi M,\mathbf{E}_p\boxtimes\mathbf{E}^*_p)}\le Cp^N\exp\Bigl(-\frac{\vr^2p}{16u}\Bigr).	  
\end{align}
By finite propagation speed for solutions of hyperbolic equations (cf. \cite[Appx. D]{MMa07}), the map $z'\mapsto\langle z\mid\tilde{F}_{u,\vr}(uB^2_p)\mid z'\rangle$ vanishes on the complement of $B^Z_\vr(z)$ and depends for any $z\in Z$ only on the restriction of the operator $B^2_p$ to the ball $B^Z_\vr(z)$. In particular, if $\dist^Z$ denotes the distance function on $Z$,
\begin{align}\label{F=0}
\langle\gamma^{-1}z\mid\tilde{F}_{\frac{u}{p},\vr}(\frac{u}{p}B^2_p)\mid z\rangle=0,\quad\text{if}\quad\dist^Z(\gamma^{-1}z,z)\ge\vr. 
\end{align} 
Let $N_{Z^\gamma/Z}$ be the normal bundle along $Z^\gamma$ which will be identified with the orthogonal bundle of $TZ^\gamma$ in $TZ_{|{Z^\gamma}}$. We will identify $Z^\gamma$ with the set of zero sections. For $\ve>0$, let $V_\ve$ be the $\ve$-neighbourhood of $Z^\gamma$ in $N_{Z^\gamma/Z}$. The injectivity radius of $Z$ shall be denoted by $\inj^Z$. There exists $\ve_0\in]0,\inj^Z/4]$ such that if $\ve\in]0,2\ve_0]$, the map $\ N_{Z^\gamma/Z}\ni (z,V)\mapsto\exp^Z_z(V)$ is a diffeomorphism of $V_\ve$ on the tubular neighbourhood $U_\ve$ of $Z^\gamma$ in $Z$. In the sequel, we will identify $V_\ve$ and $U_\ve$. We now assume that $\vr\in]0,\ve_0]$ is small enough such that for $z\in Z$, 
\begin{align}\label{dist}
z\in U_{\ve_0},\quad\text{if}\quad\dist^Z(\gamma^{-1}z,z)\le\vr.
\end{align}
Such $\vr,\ve_0$ can be chosen uniformly for $b_0$ varying in a compact subset of $B$. By (\ref{F+G=B})-(\ref{dist}), the task of calculating the asymptotic of the kernel of $\gamma e^{-\frac{u}{p}B^2_p}$ as $p\rightarrow\infty$ is localized on $U_{\ve_0}$.

After the localization we replace the manifold with the tangent space. Therefore, fix $z_0\in Z$. Let $B_\ve^Z(z_0)\subset Z$ and $B_\ve^{T_{\R,z_0}Z}(0)\subset T_{\R,z_0}Z$ be the open balls with radius $\ve>0$ and center $z_0$ and $0$ respectively. For $\ve$ sufficiently small, the exponential map $\exp_{z_0}^Z:B_\ve^{T_{\R,z_0}Z}(0)\rightarrow B_\ve^Z(z_0),~V\mapsto\exp_{z_0}^ZV$ is a diffeomorphism. We will identify  $B_\ve^Z(z_0)$ and $B_\ve^{T_{\R,z_0}Z}(0)$ by this. For $V\in B_\ve^{T_{\R,z_0}Z}(0)$, we identify $(\mathbf{E}_{p,V},h_V^{\mathbf{E}_p})$ with $(\mathbf{E}_{p,z_0},h_{z_0}^{\mathbf{E}_p})$ by parallel transport along the geodesic ray $[0,1]\ni t\mapsto tV$ with respect to the connection $\nabla^{\mathbf{E}_p}$. Let $\overline{\vt}_1,\vt^\mathcal{L}$ be the corresponding connection forms of $\overline{\nabla}_1,\nabla^\mathcal{L}$ in this trivialisation.

For convenience, as we are frequently working on $T_{\R,z_0}Z$, put
\[Z_0:=T_{\R,z_0}Z.\]
Let $g^{T_\R Z_0}$ be a Riemannian metric on $Z_0$ given by
\[g^{T_\R Z_0}:=\begin{cases}
g^{T_\R Z}&\text{on~~}B^{T_{\R,z_0}Z}_{2\ve_0}(0),\\
g_{z_0}^{T_\R Z}&\text{on~~}T_{\R,z_0}Z\setminus B^{T_{\R,z_0}Z}_{4\ve_0}(0)\end{cases}\]
and $\nabla^{T_\R Z_0,LC}$ be the Levi-Civita connection of $(Z_0,g^{T_\R Z_0})$. Let $\kappa_{z_0}\in C^\infty(Z_0)$ be the positive function with $\kappa_{z_0}(0)=1$ and
\[{\dvol_{g^{T_\R Z_0}}}_{|_V}=\kappa_{z_0}(V){\dvol_{g_{z_0}^{T_\R Z}}}_{|_V}.\]

Let $d_U$ be the ordinary differentiation operator on $T_{z_0}Z$ in direction $U$. Let $\rho:\R\rightarrow[0,1]$ be a smooth even function such that $\rho(t)=\begin{cases}
1~,~|t|<2,\\
0~,~|t|>4
\end{cases}$. We will now consider $\mathbf{E}_{p,z_0}=\Lambda^\bullet T^*_{\R,b_0}B\otimes(\Lambda^{0,\bullet}T^*Z\otimes\mathcal{E}\otimes\mathcal{L}^p)_{z_0}$ as a trivial vector bundle on $T_{z_0}Z$ which will be equipped with the Hermitian connection
\[\nabla^{\mathbf{E}_{p,z_0}}:=d+\rho\Bigl(\frac{\|V\|}{\ve_0}\Bigr)(p\vt^\mathcal{L}+\overline{\vt}_1)_{|_V}.\]
Let $\Delta^{\mathbf{E}_{p,z_0}}$ be the Bochner Laplacian associated with $\nabla^{\mathbf{E}_{p,z_0}}$ and $g^{T_\R Z_0}$. Let $\{\tilde{e}_i\}_i$ be the parallel transport of $\{e_{i,z_0}\}_i$ with respect to $\nabla^{T_\R Z_0,LC}$ along radial geodesics. Let $\Phi^\mathcal{E}\in\Gamma(Z_0,\End(\mathbf{E}_{p,z_0}))$ be given by
\begin{align*}
\Phi^\mathcal{E}:=&~\frac{s^Z}{8}+\frac{1}{4}\sum_{i,j}c(\tilde{e}_i)c(\tilde{e}_j)\Bigl(\Omega^\mathcal{E}+\frac{1}{2}\Tr\Omega^{TZ}\Bigr)(\tilde{e}_i,\tilde{e}_j)+\frac{1}{\sqrt{2}}\sum_{i,\alpha}c(\tilde{e}_i)f^\alpha\Bigl(\Omega^\mathcal{E}+\frac{1}{2}\Tr\Omega^{TZ}\Bigr)(\tilde{e}_i,f_\alpha)\\
&+\sum_{\alpha,\beta}\frac{f^\alpha f^\beta}{2}\Bigl(\Omega^\mathcal{E}+\frac{1}{2}\Tr\Omega^{TZ}\Bigr)(f_\alpha,f_\beta)+\left(\overline{\partial}^M\partial^Mi\omega^M\right)^c+\frac{1}{16}\left\|(\overline{\partial}^Z-\partial^Z)\omega^Z\right\|_{\Lambda^\bullet T^*_\R Z}.
\end{align*}
Define the operator
\begin{align*}
M_{p,z_0}:=&~\frac{1}{2}\Delta^{\mathbf{E}_{p,z_0}}+\rho\Bigl(\frac{\|V\|}{\ve_0}\Bigr)\Phi^\mathcal{E}_{|_V}+p\rho\Bigl(\frac{\|V\|}{\ve_0}\Bigr)\biggl(\frac{1}{4}\sum_{i,j}c(\tilde{e}_i)c(\tilde{e}_j)\Omega^\mathcal{L}(\tilde{e}_i,\tilde{e}_j)\\
&+\frac{1}{\sqrt{2}}\sum_{i,\alpha}c(\tilde{e}_i)f^\alpha\Omega^\mathcal{L}(\tilde{e}_i,f_\alpha)+\sum_{\alpha,\beta}\frac{f^\alpha f^\beta}{2}\Omega^\mathcal{L}(f_\alpha,f_\beta)\biggr)_{|_V}
\end{align*}
which is a second order elliptic differential operator acting on $\Gamma(Z_0,\mathbf{E}_{p,z_0})$ and coincides with $B^2_p$ over $B_{2\ve_0}^{T_{z_0}Z}(0)$ by the Bismut-Lichnerowicz formula (\ref{Lichnerowicz}). Let $\bigl\langle V\bigm|e^{-M_{p,z_0}}\bigm|V'\bigr\rangle$ be the smooth kernel of $M_{p,z_0}$ with respect to $\dvol_{g^{T_\R Z_0}}(V')$. Observe that if $V\in N_{Z^\gamma/Z,\R,z_0}$ with $\|V\|\le\ve_0$ and if $z'\in Z$ with $\dist^Z(V,z')\le\vr$, then $\dist^Z(z,z')\le2\ve_0$. In particular, $z'$ is represented by $V'\in Z_0$ such that $\|V'\|\le2\ve_0$, so that $\rho(V')=1$.

By finite propagation speed for solutions of hyperbolic equations, we have for $V\in N_{Z^\gamma/Z,\R,z_0}$ with $\|V\|\le\ve_0$,
\begin{align}\label{äqwelle}
\langle\gamma^{-1}V\mid\tilde{F}_{\frac{u}{p},\vr}(\frac{u}{p}B^2_p)\mid V\rangle=\langle\gamma^{-1}V\mid\tilde{F}_{\frac{u}{p},\vr}(\frac{u}{p}M_{p,z_0})\mid V\rangle.
\end{align}

\subsection[Asymptotic Expansion for \texorpdfstring{$0<u\le1$}{0<u=<1}]{Asymptotic Expansion for \texorpdfstring{$\boldsymbol{0<u\le1}$}{0<u=<1}}\label{subsec 2.2}
We now do a change of the parameter. Set $t:=\frac{1}{\sqrt{p}}\in ]0,1]$. For $s\in\Gamma(Z_0,\mathbf{E}_{p,z_0})$ and $V\in Z_0$, set
\begin{align}\label{rescale}
&(S_ts)_{|_V}:=s_{|_{V/t}},\notag\\
&L^t_{z_0}:=t^2S_t^{-1}\kappa_{z_0}^{1/2}M_{p,z_0}\kappa_{z_0}^{-1/2}S_t.
\end{align}
If $\mathrm{Op}(Z_0,\mathbf{E}_{p,z_0})$ denotes the set of scalar differential operators acting on $\Gamma(Z_0,\mathbf{E}_{p,z_0})$, then 
\[L_{z_0}^t\in\End(\mathbf{E}_{p,z_0})\otimes\mathrm{Op}(Z_0,\mathbf{E}_{p,z_0}).\]
Let $\bigl\langle V\bigm|e^{-L^t_{z_0}}\bigm|V'\bigr\rangle$ be the smooth kernel of the operator $e^{-L^t_{z_0}}$ with respect to $\dvol_{g_{z_0}^{T_\R Z}}(V')$. Denote by $\pi_{\scalebox{0.6}{$\displaystyle{T_\R Z}$}}: T_\R Z\rightarrow M$ the submersion of the vector bundle $T_\R Z$ over $M$. When $z_0$ varies on $Z$,~$\bigl\langle\cdot\bigm|e^{-L^t}\bigm|\cdot\bigr\rangle$ can be viewed as a section of $\pi_{\scalebox{0.6}{$\displaystyle{T_\R Z}$}}^*\End(\mathbf{E}_p)$ over $T_\R Z\times_{\pi_{\scalebox{0.6}{$\displaystyle{T_\R Z}$}}}T_\R Z$. Similar for $M_{p,z_0}$, we write $M_p$ when $b\in B, z_0\in Z_b$ vary.

Let $\dot\Omega^{Z,\mathcal{L}}\in\Gamma(M,\End(T^{1,0}Z))$ be the Hermitian matrix defined by, if, $X,Y\in \Gamma(M,T^{1,0}Z)$, 
\[\Omega^\mathcal{L}(X,\overline{Y})=:h^{T^{1,0}Z}(\dot\Omega^{Z,\mathcal{L}}X,Y).\]
Then $\dot\Omega^{Z,\mathcal{L}}$ is positive definite by the assumption on positivity of $\mathcal{L}$. This endomorphism extends to $\End(T_\C Z)$ via $\dot\Omega^{Z,\mathcal{L}}\overline{w}_i=-\overline{\dot\Omega^{Z,\mathcal{L}}w_i}$. For $u>0$, $z\in Z^\gamma$ and $V\in N_{Z^\gamma/Z,\R,z}$, put
\begin{align*}
&\Omega_u:=u\sum_{k,l}\Omega^\mathcal{L}(w_k,\overline{w}_l)\overline{w}^l\wedge\iota_{\overline{w}^k}+\sqrt{\frac{u}{2}}\sum_{i,\alpha}c(e_i)f^\alpha\Omega^\mathcal{L}(e_i,f_\alpha)+\sum_{\alpha,\beta}\frac{f^\alpha f^\beta}{2}\Omega^\mathcal{L}(f_\alpha,f_\beta),\\
&I_u(z,V):=\frac{1}{(2\pi)^n}\exp(-\Omega_{u,(z,V)})\frac{\det\dot{\Omega}^\mathcal{L}_{(z,V)}}{\det\bigl(\id_{T^{1,0}Z}-\exp(-u\dot{\Omega}^\mathcal{L}_{(z,V)})\bigr)},\\
&e_{\gamma,u}(z,V):=\exp\Bigl(
\begin{aligned}[t]
&-g^{T_\C Z}\Bigl(\frac{\dot{\Omega}^\mathcal{L}_{(z,V)}/2}{\tanh(u\dot{\Omega}^\mathcal{L}_{(z,V)}/2)}(z,V),(z,V)\Bigr)\\
&+g^{T_\C Z}\Bigl(\frac{\dot{\Omega}^\mathcal{L}_{(z,V)}/2}{\sinh(u\dot{\Omega}^\mathcal{L}_{(z,V)}/2)}e^{u\dot{\Omega}^\mathcal{L}_{(z,V)}/2}(z,\gamma^{-1}V),(z,V)\Bigr)\Bigr).
\end{aligned}
\end{align*}
Let $\dvol_{g^{T_\R Z^\gamma}}$ and $\dvol_{g^{N_{Z^\gamma/Z,\R}}}$ be the  volume forms on $TZ^\gamma$ and $N_{Z^\gamma/Z}$ induced by $h^{TZ}$. Let $\hat{\kappa}\in C^\infty(V_{\ve_0})$ be the positive function defined by
\begin{align}\label{kappa hat}
{\dvol_{g^{T_\R Z}}}_{|_{(z,V)}}=\hat{\kappa}(z,V){\dvol_{g^{T_\R Z^\gamma}}}_{|_z}{\dvol_{g^{N_{Z^\gamma/Z,\R}}}}_{|_V}.
\end{align}
Then $\hat{\kappa}(z,0)=1$ and $\hat{\kappa}(\cdot,V)_{|_{Z^\gamma}}\equiv1$.

Let $\{Z^\gamma_l\}_{l\in L}$ be the connected components of $Z^\gamma$ with $n_{\gamma,l}:=\dim_\C Z^\gamma_l$. The action of $\gamma^\mathcal{L}$ on $\mathcal{L}$ is given by multiplication with $e^{i\theta}$ which is locally constant on $M^\gamma$ and hence constant on the components $Z^\gamma_l$.
\begin{Lemma}\label{äqkern}
Let $m\in\N_0$. There exists $\delta>0$ such that as $p\rightarrow\infty$, uniformly as u varies in a compact subset of $\R_{>0}$, the following asymptotic for the $C^m(B)$-norm holds:
\begin{align*}
&\psi_{1/\sqrt{p}}\Tr_s\Bigl[\gamma N_{u/p}\exp\bigl(-B^2_{p,u/p}\bigr)\Bigr]\\
=~&\rk(\mathcal{E})\sum_{l\in L}p^{n_{\gamma,l}}e^{ip\theta}_{|_{Z_l^\gamma}}\int_{z\in Z_l^\gamma}\int\limits_{\substack{V\in N_{Z_l^\gamma/Z,\R,z}}}\begin{aligned}[t] &e_{\gamma,u}(z,V)\Tr_s\Bigl[\gamma^{\mathcal{E}}N_uI_u(z,V)\Bigr]\\&\qquad\cdot\dvol_{g^{N_{Z^\gamma/Z,\R}}}(V)\dvol_{g^{T_\R Z^\gamma}}(z)+o(p^{n_{\gamma,l}-\delta}).
\end{aligned}
\end{align*}
\end{Lemma}
\begin{proof} Using $\psi_{1/\sqrt{p}}N_{u/p}=N_u$ and $\psi_{1/\sqrt{p}}e^{-B^2_{p,u/p}}=\psi_{1/\sqrt{u}}e^{-\frac{u}{p}B^2_p}$ followed by (\ref{F+G=B}), (\ref{G Abschätzung}) and the discussion after it, where $U_{\ve_0}, V_{\ve_0}$ are identified, we find for any $k\in\N$, as $p\rightarrow\infty$,
\begin{align*}
&\psi_{1/\sqrt{p}}\Tr_s\Bigl[\gamma N_{u/p}\exp\bigl(-B^2_{p,u/p}\bigr)\Bigr]\\
=&\int\limits_Z\Tr_s\Bigl[\gamma N_u\psi_{1/\sqrt{u}}\bigl\langle\gamma^{-1}z\bigm|\exp\Bigl(-\frac{u}{p}B^2_p\Bigr)\bigm|z\bigr\rangle\Bigr]\,\dvol_{g^{T_\R Z}}(z)\\
=&\int\limits_{V_{\ve_0}}\Tr_s\Bigl[\gamma N_u\psi_{1/\sqrt{u}}\bigl\langle\gamma^{-1}(z,V)\bigm|\tilde{F}_{\frac{u}{p},\vr}\Bigl(-\frac{u}{p}B^2_p\Bigr)\bigm|(z,V)\bigr\rangle\Bigr]\,\dvol_{g^{T_\R Z}}(z,V)+o(p^{-k}).
\end{align*}
Using (\ref{äqwelle}) and (\ref{kappa hat}) we insert
\begin{align*}
&\int\limits_{V_{\ve_0}}\Tr_s\Bigl[\gamma N_u\psi_{1/\sqrt{u}}\bigl\langle\gamma^{-1}(z,V)\bigm|\tilde{F}_{\frac{u}{p},\vr}\Bigl(-\frac{u}{p}B^2_p\Bigr)\bigm|(z,V)\bigr\rangle\Bigr]\dvol_{g^{T_\R Z}}(z,V)\\
=&\int\limits_{V_{\ve_0}}\Tr_s\Bigl[\gamma N_u\psi_{1/\sqrt{u}}\bigl\langle\gamma^{-1}(z,V)\bigm|\tilde{F}_{\frac{u}{p},\vr}\Bigl(-\frac{u}{p}M_p\Bigr)\bigm|(z,V)\bigr\rangle\Bigr]\hat{\kappa}(z,V)\,\dvol_{g^{N_{Z^\gamma/Z,\R}}}(V)\dvol_{g^{T_\R Z^\gamma}}(z).
\end{align*}
Because both $B^2_p$ and $M_p$ have the same structure, (\ref{G Abschätzung}) is also true for $M_p$. Furthermore, the following identity holds, 
\[\bigl\langle(z,V)\bigm|e^{-\frac{u}pM_p}\bigm|(z,V')\bigr\rangle=p^n\bigl\langle(z,V/t)\bigm|e^{-uL^t}\bigm|(z,V'/t)\bigr\rangle\kappa^{-1/2}_z(V)\kappa^{-1/2}_z(V').\]
Thus, for any $k\in \N$, we conclude
\begin{align}\label{äq0}
&\psi_{1/\sqrt{p}}\Tr_s\Bigl[\gamma N_{u/p}\exp\bigl(-B^2_{p,u/p}\bigr)\Bigr]\nonumber\\
=~&\int_{z\in Z^\gamma}\int\limits_{\substack{V\in N_{Z^\gamma/Z,\R,z}\\\| V \|\le\ve_0}}\Tr_s\Bigl[\gamma N_u\psi_{1/\sqrt{u}}\Bigl\langle(z,\gamma^{-1}V)\Bigm|\exp\Bigl(-\frac{u}{p}M_p\Bigr)\Bigm|(z,V)\Bigr\rangle\Bigr]\nonumber\\
&\quad\cdot\hat{\kappa}(z,V)\,\dvol_{g^{N_{Z^\gamma/Z,\R}}}(V)\dvol_{g^{T_\R Z^\gamma}}(z)+o(p^{-k})\nonumber\\
=~&p^n\int_{z\in Z^\gamma}\int\limits_{\substack{V\in N_{Z^\gamma/Z,\R,z}\\\|V\|\le\ve_0}}\Tr_s\Bigl[\gamma N_u\psi_{1/\sqrt{u}}\Bigl\langle(z,\gamma^{-1}V/t)\Bigm|e^{-uL^t}\Bigm|(z,V/t)\Bigr\rangle\Bigr]\nonumber\\
&\quad\cdot\kappa_z^{-1/2}(\gamma^{-1}V)\kappa_z^{-1/2}(V)\hat{\kappa}(z,V)\,\dvol_{g^{N_{Z^\gamma/Z,\R}}}(V)\dvol_{g^{T_\R Z^\gamma}}(z)+o(p^{-k})\nonumber\\
=~&\sum_{l\in L}p^{n_{\gamma,l}}\int_{z\in Z_l^\gamma}\int\limits_{\substack{V\in N_{Z_l^\gamma/Z,\R,z}\\\|tV\|\le\ve_0}}\Tr_s\Bigl[\gamma N_u\psi_{1/\sqrt{u}}\Bigl\langle (z,\gamma^{-1}V)\Bigm|e^{-uL^t}\Bigm|(z,V)\Bigr\rangle\Bigr]\nonumber\\
&\quad\cdot\kappa_z^{-1/2}(\gamma^{-1}tV)\kappa_z^{-1/2}(tV)\hat{\kappa}(z,tV)\,\dvol_{g^{N_{Z^\gamma/Z,\R}}}(V)\dvol_{g^{T_\R Z^\gamma}}(z)+o(p^{-k}),
\end{align}
where in the last step we used the transformation $V\mapsto tV$. For $u>0$, let $L^0_u$ be the harmonic oscillator given by
\[L^0_u:=\frac{u}{2}\sum_{i=1}^{2n}\Bigl(d+\frac{1}{2}g^{T_\C Z}\bigl(\dot{\Omega}_z^{Z,\mathcal{L}}V,e_i\bigr)\Bigr)^2+\Omega_u(z)-\frac{u}{2}\Tr\dot{\Omega}_z^{T,\mathcal{L}}.\]
By \cite[(3.127),(3.138)-(3.142)]{P23} (or  \cite[(2.85)-(2.87)]{P15}), it follows that 
\begin{align}\label{äq1}
\psi_{1/\sqrt{u}}\Bigl\langle(z,\gamma^{-1}V)\Bigm|e^{-uL^t_z}\Bigm|(z,V)\Bigr\rangle=\Bigl\langle(z,\gamma^{-1}V)\Bigm|e^{-L^0_u}\Bigm|(z,V)\Bigr\rangle+O(p^{-1/(4n+2)})
\end{align}
for the $C^m$-norm on ${\Gamma(M\times_\pi M,\mathbf{E}_p\boxtimes\mathbf{E}^*_p)}$.
Since $\gamma$ is an isometry, the formula for the heat kernel of a harmonic oscillator (cf. \cite[E.2.4]{MMa07}) gives
\begin{align}\label{äq2}
&\Bigl\langle(z,\gamma^{-1}V)\Bigm|e^{-L^0_u}\Bigm|(z,V)\Bigr\rangle= e_{\gamma,u}(z,V)\cdot I_u(z,V)\otimes\id_{\mathcal{E}_{(z,V)}}.
\end{align}
By (\ref{äq0}), (\ref{äq1}), (\ref{äq2}), the Taylor expansion of $\kappa_z,\hat{\kappa}$ and because $e_{\gamma,u}(z,V)$ decays exponentially as $\|V\|\rightarrow\infty$, we conclude the claim.
\end{proof}
\begin{Remark} The positivity assumption is not needed when we use the convention, that if an eigenvalue of $\dot{\Omega}^{Z,\mathcal{L}}$ is zero, then its contribution to $e_{\gamma,u}(z,V)I_u(z,V)\otimes\id_{\mathcal{E}_{(z,V)}}$ is $X\mapsto\frac{1}{2\pi u}e^{-\frac{1}{2u}\|\gamma^{-1}X-X\|^2}$.
\end{Remark}
The function $I_u(z,V)$ has an asymptotic expansion as $u\searrow0$ (\cite[(3.145)]{P23}) and because $e_{\gamma,u}(z,V)$ is analytic in $u=0,$ we find sections $A^{[j]}_\gamma\in\Gamma(Z,\End(\Lambda^\bullet T^*_{\R,{b_0}}B\otimes\Lambda^{0,\bullet}T^*Z))$ for $j\ge-d$ such that as $u\searrow0$,
\[\bigl(e_{\gamma,u}\cdot I_u\bigr)(z,V)=\sum_{j=-d}^kA^{[j]}_\gamma(z,V)u^j+O(u^{k+1}).\]
Because $Z$ is compact, $L$ is finite and we can set
\[n_\gamma:=\max_{l\in L}n_{\gamma,l},\quad L_0:=\{l\in L\mid n_{\gamma,l}=n_\gamma\}.\]
Put $A^{[-d-1]}_\gamma:=0$ and for $j\ge-d-1 $, set
\begin{align*}
&\underline{b}_{\gamma,j,l}:=\int_{z\in Z^\gamma_l}\int\limits_{\substack{V\in N_{Z_l^\gamma/Z,\R,z}}}\Tr_s\Bigl[\gamma^{\mathcal E}\bigl(N_VA_\gamma^{[j]}(z,V)-i\omega^HA_\gamma^{[j+1]}(z,V)\bigr)\Bigr]\\
&\hspace{11em}\cdot\dvol_{g^{N_{Z^\gamma/Z,\R}}}(V)\dvol_{g^{T_\R Z^\gamma}}(z),\\
&b_{\gamma,j,e^{ip\theta}}:=\sum_{l\in L_0}e^{ip\theta}_{|_{Z_l^\gamma}} \cdot \underline{b}_{\gamma,j,l}.
\end{align*}
Then for any $k\in\N_0$, as $u\searrow0$,
\begin{align}\label{äq3}
&\sum_{l\in L_0}e^{ip\theta}_{|_{Z_l^\gamma}}\int_{z\in Z^\gamma_l}\int\limits_{\substack{V\in N_{Z_l^\gamma/Z,\R,z}}}e_{\gamma,u}(z,V)\Tr_s\Bigl[\gamma^{\mathcal{E}}N_uI_u(z,V)\Bigr]\,\dvol_{g^{N_{Z^\gamma/Z,\R}}}(V)\dvol_{g^{T_\R Z^\gamma}}(z)\notag\\
=&\sum_{j=-d-1}^kb_{\gamma,j,e^{ip\theta}}u^j+O(u^{k+1}).
\end{align}
\begin{Lemma}\label{äqcor}
There exist forms $\tilde{b}_{\gamma,p,j}\in\mathfrak{A}^\bullet(B)$ such that for any $k,m\in\N_0$, there exists $C>0$ such that for any $u\in]0,1]$ and $p\in\N$,
\[\Bigl\|p^{-n_\gamma}\psi_{1/\sqrt{p}}\Tr_s\Bigl[\gamma N_{u/p}\exp\bigl(-B^2_{p,u/p}\bigr)\Bigr]-\sum_{j=-d-1}^k\tilde{b}_{\gamma,p,j}u^j\Bigr\|_{C^m(B,\Lambda^\bullet T^*_\C B)}\le Cu^{k+1}.\]
Moreover, as $p\rightarrow\infty$ for any $j\ge-n_\gamma$,
\[\tilde{b}_{\gamma,p,j}=\rk(\mathcal{E}) b_{\gamma,j,e^{ip\theta}}+O\Bigl(\frac{1}{\sqrt{p}}\Bigr),\]
where the convergence is in the $C^\infty$ topology on $B$.
\end{Lemma}
\begin{proof} For sake of completeness, we reproduce the necessary steps for the proof of this result given in \cite[Thm. 2.21, Cor. 2.22]{P15} and show the off-diagonal adjustments.

As in Subsection \ref{subsec 2.1}, the problem can be localized near $z_0\in Z$ and we rescale the superconnection as in (\ref{rescale}) to obtain the operator $p^{n-n_{\gamma}}L^t_{z_0}$. By the finite propagation speed of the wave operator \cite[Thm. D.2.1]{MMa07}, for $t$ small, $\langle0\mid\tilde{F}_{u,\ve_0}(uL^t_{z_0})\mid\cdot\rangle $ only depend on the restriction of $L^t_{z_0}$ on $B^{T_{\R,z_0}Z}_{2\ve_0}(0)$ and is supported in $B^{T_{\R,z_0}Z}_{2\ve_0}(0)$. Consider the sphere bundle 
\[\mathbb{S}:=\{(V,r)\in T_\R Z\times\R\mid\|V\|^2+r^2=1\}\] 
over $Z$. The ball $B^{T_{\R,z_0}Z}_{2\ve_0}(0)$ will be embedded in $\mathbb{S}_{z_0}$ by the map 
\[V\mapsto\bigl(V,\sqrt{1-\|V\|^2}\bigr)\]
and the operator $L^t_{z_0}$ will be extended to a generalized Laplacian $\tilde{L}_{t,z_0}$ on $\mathbb{S}_{z_0}$ with values in $\pi_{\scalebox{0.6}{$\displaystyle{T_\R Z}$}}^*\End(\mathbf{E}_p)$. By \cite[Prop. 2.2, (2.54)]{P15}, we see that \cite[(2.93)]{P15} still holds outside of the diagonal, i.e. we have
\begin{align*}
&\Bigl\|\bigl\langle\gamma^{-1}V\bigm|e^{-u L^t_z}-e^{-u \tilde{L}^t_z}\bigm|V\bigr\rangle\Bigr\|_{C^m(M\times[0,1],\End(\mathbf{E}_p))}\le C\exp\Bigl(-\frac{\ve_0^2}{32u}\Bigr).
\end{align*}
As the total space of $\mathbb{S}$ is compact, the heat kernel $\bigl\langle\gamma^{-1}V\bigm|e^{-u\tilde{L}_t}\bigm|V\bigr\rangle$ has an asymptotic expansion as $u\searrow0$ starting with $u^{-n_\gamma}$ which depends smoothly on the parameters $z_0$ and $t$ (\cite[Ch. 6]{BeGeV04}). From this, we get the first statement and the second statement follows from (\ref{äq1}), (\ref{äq2}) and (\ref{äq3}).
\end{proof}

\subsection{Large Parameter Estimates and Proof of Theorem 1}\label{subsec 2.3}
The first goal of this subsection is to prove that the norm of $p^{-n_\gamma}\psi_{1/\sqrt{p}}\Tr_s\bigl[\gamma N_{u/p}\exp({-B^2_{p,u/p}})\bigr]$ for $u\ge1$ is bounded by $\frac{C}{\sqrt{u}}$ for a constant $C>0$ which will permit us to use Lebesgue's dominated convergence theorem in the prove of Theorem \ref{Th1}. For $A\in\mathfrak{A}^\bullet(B,\{\End(\mathfrak{A}^{0,\bullet}(Z_b,\mathcal{E}_b\otimes\mathcal{L}_b^p))\}_{b\in B})$ and $q\ge1$, let $\|A\|_q$ be the Schatten $q$-norm given by
\[\|A\|_q:=\Bigl(\Tr\Bigl[(A^*A)^{q/2}\Bigr]\Bigr)^{1/q}.\]
Because $\gamma$ is an unitary operator and the Schatten norms are unitary invariant, we have $\|\gamma A\|_q=\|A\|_q$. Moreover, as the operators we are considering commutes with $\gamma$, the arguments showing boundedness remains the same as in the non-equivariant case of \cite[Sec. 3.9]{P23} and we will not repeat those arguments if not necessary. Puchol's paper \cite{P23} already provides all the analytical tools that are needed. The rather unclear steps will be those where the $p^{-n}$ are replaced by $p^{-n_\gamma}$ which will be explained in greater detail.

By the spectral gap property, there exists $\vartheta >0$ such that for $p$ suffieciently large,
\begin{align}\label{spectrum}
\Sp\big(B^2_p/p\big)\subset \{0\}\cup[\vartheta,\infty[.
\end{align}
\noindent Let $\delta+\Delta$ be the contour in $\C$ indicated in the following figure: 
\begin{figure}[H]
\begin{center}
\begin{tikzpicture}
\draw[->] (-4,0) -- (3,0);
\draw[->] (-2,-1.8) -- (-2,1.8);
\draw[line width=1.2,-](-0.5,-1) -- (-0.5,1);	
\draw[line width=1.2,-](-0.5,-1) -- (3,-1);	
\draw[line width=1.2,-](-0.5,1) -- (3,1);	
\draw[decoration={markings,mark=at position 0.4 with {\arrow{latex}},
mark=at position 0.9 with {\arrow{latex}}},postaction={decorate},line width=1.2] (-2,0) circle (0.8cm);
\filldraw (-2,0) circle (1pt);
\filldraw (-2,1) circle (1pt);
\filldraw (-2,-1) circle (1pt);
\filldraw (-1.2,0) circle (1pt);
\filldraw (-0.5,0) circle (1pt);
\draw[,-{Latex[scale=1.2]}] (1.5,1) -- (1.4,1);
\draw[,-{Latex[scale=1.2]}] (0,-1) -- (0.5,-1);
\draw[] (3.3,0) node {$\R$};
\draw[] (-2,2.1) node {$i\R$};
\draw[] (-1.8,-0.25) node {$0$};
\draw[] (-2.2,1.2) node {$i$};
\draw[] (-2.35,-1.2) node {$-i$};
\draw[] (-1.05,-0.3) node {$\frac{\vt}{2}$};
\draw[] (-0.3,-0.3) node {$\frac{\vt}{4}$};
\draw[] (1.5,-0.7) node {$\Delta$};
\draw[] (-3,-0.5) node {$\delta$};
\end{tikzpicture}
\end{center}
\caption{Contour $\delta+\Delta$.}
\end{figure}
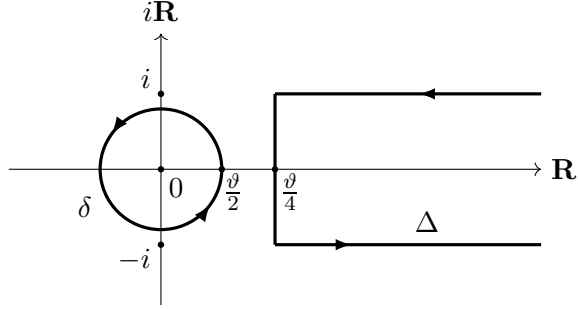
\noindent Define the operators
\begin{align*}
&\mathbb{K}_{p,u}:=\frac{1}{2\pi i}\psi_{1/\sqrt{u}}\int_\Delta e^{-u\lambda}(\lambda-B^2_p/p)^{-1}\,d\lambda,\\
&\mathbb{P}_{p,u}:=\frac{1}{2\pi i}\psi_{1/\sqrt{u}}\int_\delta e^{-u\lambda}(\lambda-B^2_p/p)^{-1}\,d\lambda.
\end{align*}

\begin{Lemma}\label{large param} For any $k,m\in\N_0$, there are $c,C_1,C_2,C_3>0$ such that for $u\ge1$ and $p\in\N$:
\begin{enumerate}[label=\alph*)]
\item $\Bigl\|p^{-n_\gamma}\Tr_s\Bigl[\gamma N_u\mathbb{K}_{p,u}\Bigr]\Bigr\|_{C^m(B,\Lambda^\bullet T^*_\C B)}
\le C_1e^{-cu},$
\item $\Bigl\|p^{-n_\gamma}\Tr_s\Bigl[\gamma N_u\mathbb{P}_{p,u}\Bigr]\Bigr\|_{C^m(B,\Lambda^\bullet T^*_\C B)}\le\frac{C_2}{\sqrt{u}},$
\item $\Bigl\|p^{-n_\gamma}\psi_{1/\sqrt{p}}\Tr_s\Bigl[\gamma N_{u/p}\exp({-B^2_{p,u/p}})\Bigr]\Bigr\|_{C^m(B,\Lambda^\bullet T^*_\C B)}\le\frac{C_3}{\sqrt{u}}.$
\end{enumerate}
\end{Lemma}
\begin{proof} The constants appearing in the proof will often be denoted by the same letter $C$ even though their value differ from one equation to another. In the proof we will assume without loss of generality that in the sequel, $\Sp(B^2_p/p)\subset\{0\}\cup[\vt,\infty]$ holds for $p\in\N$.\\
$a):$ Set
\[H_p:=(\overline{\partial}_p+\overline{\partial}_p^*)^2/p-\lambda_0.\]
From \cite[(3.242), (3.243)]{P23}, for $k\gg1$ large enough,
\begin{align}\label{äqint}
\Tr\bigl[\gamma H_p^{-k}\bigr]=-\frac{1}{(k-1)!}\int_0^\infty\Tr\bigl[\gamma e^{-tH_p}\bigr]t^{k-1}\,dt.
\end{align}
By Lemma \ref{äqkern} in degree zero, $p^{-n_\gamma}\Tr\bigl[\gamma e^{(\overline{\partial}_p+\overline{\partial}_p^*)^2/p}\bigr]$ and its derivatives are bounded. Thus, as in \cite[(3.244)]{P23}, for $m\in\N_0$, there is $C>0$ such that for $t\ge1$ and $p\in\N$,
\begin{align}\label{äqa}
p^{-n_\gamma}\Bigl\|\Tr\bigl[\gamma H_p^{-k}\bigr]\Bigr\|_{C^m(B)}\le p^{-n_\gamma}\Bigl\|\Tr\bigl[\gamma e^{(\overline{\partial}_p+\overline{\partial}_p^*)^2/p}\bigr]\Bigr\|_{C^m(B)}e^{\lambda_0t}\le Ce^{\lambda_0 t}. 
\end{align}
In the same way as in \cite[(3.245)]{P23}, we find by using Lemma \ref{äqcor} in degree $0$ that for any $k,m\in\N_0$ there exist $a_{\gamma,p,j}\in\R$ and $C>0$ such that for any $t\in]0,1]$ and $p\in\N$,
\begin{align}\label{äqb}
\Bigl\|p^{-n_\gamma}\Tr_s\Bigl[\gamma e^{-\frac{t}{p}(\overline{\partial}_p+\overline{\partial}_p^*)^2}\Bigr]-\sum_{j=-n_\gamma-1}^ka_{\gamma,p,j}t^j\Bigr\|_{C^m(B)}\le Ct^{k+1}.
\end{align}
Splitting the integral in (\ref{äqint}) at $t=1$ and using (\ref{äqa}), (\ref{äqb}), we get for $k$ large enough,
\[p^{-n_\gamma}\Bigl\|\Tr\bigl[\gamma H_p^{-k}\bigr]\Bigr\|_{C^m(B)}\le C.\]
Thus, there exists $q_0\in\N_0$ such that for $q>q_0$, there is $C>0$ such that
\[p^{-n_\gamma}\Bigl\|\gamma(\lambda_0-B^2_p/p)^{-q}\Bigr\|_1\le p^{-n_\gamma}\Bigl\|\Tr\bigl[\gamma H_p^{-q}\bigr]\Bigr\|_{C^m(B)}\le C.\]
With the same arguments as in the proof of \cite[Lemma 3.42]{P23}, we obtain its equivariant version, i.e. for $\lambda_0\in\R_{>0}$, there exists $q_0\in\N_0$ such that for $q\ge q_0$, $U\in T_\R B$ and $m\in\N_0$, there exists $C>0$ such that for $p\in\N$,
\begin{align}\label{äqLemma}
p^{-n_\gamma}\Bigl\|\bigl(\nabla_U^{\End(\mathbf{E}_p)})^m\gamma(\lambda_0-B^2_p/p)^{-q}\Bigr\|_1\le C,
\end{align}
where we use that $\nabla_U^{\End(\mathbf{E}_p)}$ is a $\gamma$-invariant connection for the case $m\ge1$. Now the proof proceeds as in \cite[Prop. 3.43]{P23} by using (\ref{äqLemma}) instead of \cite[Lemma 3.42]{P23}.\\[0.4em]
$b):$ Recall that the notation $B_p^{2,(+)}$ was explained in Section \ref{sec 0}. For $b\in B$, let $P_{p,b}$ be the orthogonal projection from $\mathfrak{A}^{0,\bullet}(Z_b,\mathcal{E}_b\otimes\mathcal{L}_b^p)$ to $\ker(\overline{\partial}_p^{Z_b}+\overline{\partial}_p^{Z_b*})$ and put $P_p:=\{P_{p,b}\}_{b\in B}$. By proceeding as in \cite[(3.263)-(3.265)]{P23}, the operator $\gamma\mathbb{P}_{p,u}$ can be rewritten as
\begin{align*}
\gamma\mathbb{P}_{p,u}&=\frac{1}{2\pi i}\psi_{1/\sqrt{u}}\int_\delta\gamma e^{-\lambda}(\lambda-\frac{u}{p}B^2_p)^{-1}\,d\lambda\nonumber\\
&=\psi_{1/\sqrt{u}}\sum_{l=0}^{\dim_\R B}\sum\limits_{\substack{1\le i_0\le l+1\\j_1,\dots,j_{l+1-i_0}\ge0\\\sum_{m=1}^{l+1-i_0}j_m\le i_0-1}}\frac{(-1)^{l-\sum_mj_m}}{(i_0-1-\sum_mj_m)!}T_{p,1}\Big(\frac{u}{p}B^{2,(+)}_p\Big)T_{p,2}\ldots\Big(\frac{u}{p}B^{2,(+)}_p\Big)T_{p,l+1}, 
\end{align*}
where $\gamma P_p$ appears $i_0$ times among the $T_{p,j}$ and the other terms are given, respectively, by $(\frac{u}{p}(\overline{\partial}_p+\overline{\partial}_p^*)^2)^{-(1+j_1)},\ldots,(\frac{u}{p}(\overline{\partial}_p+\overline{\partial}_p^*)^2)^{-(1+j_{l+1-i_0})}$.
From \cite[(3.271)]{P23}, each term in the sum is a product of uniformly bounded terms in which $\gamma P_p$ appears since $i_0\ge1$. By using the Atiyah-Segal-Singer index formula in \cite[(3.272)]{P23} instead, which implies that $\Tr_s[\gamma_{|_{\ker(\overline{\partial_p}+\overline{\partial_p}^*)}}]$ is a polynomial in $p$ of degree $n_\gamma$, and proceeding in the same way, the claim follows for $m=0$. Again, by the equivariance of the connections the case $m\ge1$ follows from the same reasoning as in the non-equivariant case.\\
$c):$ Since we have
\begin{align*}
\psi_{1/\sqrt{p}}\Tr_s\Bigl[\gamma N_{u/p}e^{-B^2_{p,u/p}}\Bigr]=\Tr_s\Bigl[\gamma N_u\psi_{1/\sqrt{u}}\bigl(e^{-\frac{u}{p}B^2_p}\bigr)\Bigr]=\Tr_s\bigl[\gamma N_u(\mathbb{P}_{p,u}+\mathbb{K}_{p,u})\bigr],
\end{align*}
the statement follows directly from $a)$ and $b)$.
\end{proof}
\begin{proof}[\textit{Proof of Theorem \ref{Th1}}] By the assumption from the Introduction, we have $H^k(Z,(\mathcal{E}\otimes\mathcal{L}^p)_{|_Z})=0$ for $k>0$ and $p\ge p_0$. In the proof we only consider such $p\ge p_0$. Define formally
\begin{align*}
&\tilde{\zeta}_{\gamma,p}(s):=-\frac{p^{-n_\gamma}}{\Gamma(s)}\int_0^\infty u^{s-1}\psi_{1/\sqrt{p}}\vp\bigr(\Tr_s\bigl[\gamma N_{u/p}\exp(-B^2_{p,u/p})\bigr])\,du.
\end{align*}
In the same fashion as in the discussion before Definition \ref{holom Torsion}, $\tilde{\zeta}_{\gamma,p}$ is a well-defined holomorphic function in $s$ near $0$. It satisfies
\begin{align}\label{proof thm 1 a}
p^{-n_\gamma}\psi_{1/\sqrt{p}}\zeta_{\gamma,p}(s)&=p^{-s}\tilde{\zeta}_{\gamma,p}(s),\nonumber\\
p^{-n_\gamma}\psi_{1/\sqrt{p}}\zeta'_{\gamma,p}(0)&=-\log(p)\tilde{\zeta}_{\gamma,p}(0)+\tilde{\zeta}_{\gamma,p}'(0).
\end{align}
On the other hand, with Lemma \ref{äqcor} we have 
\begin{align}\label{proof thm 1 b}
\tilde{\zeta}'_{\gamma,p}(0)=&-\int_0^1p^{-n_\gamma}\vp\Bigr(\psi_{1/\sqrt{p}}\Tr_s\bigl[\gamma N_{u/p}\exp(-B^2_{p,u/p})]-\sum_{j=-d-1}^0\tilde{b}_{\gamma,p,j}u^j\Bigr)\,\frac{du}{u}\nonumber\\
&-\int_1^\infty p^{-n_\gamma}\vp\psi_{1/\sqrt{p}}\Tr_s\bigl[\gamma N_{u/p}\exp(-B^2_{p,u/p})]\,\frac{du}{u}-\sum_{j=-d-1}^{-1}\frac{\tilde{b}_{\gamma,p,j}}{j}+\Gamma'(1)\tilde{b}_{\gamma,p,0},\\\label{proof thm 1 c}
\tilde{\zeta}_{\gamma,p}(0)=&-\vp \tilde{b}_{\gamma,p,0}.
\end{align}
Now set
\begin{align*}
\underline{\nu}_{\gamma,u,l}:&=\int_{z\in Z^\gamma_l}\int\limits_{\substack{V\in N_{Z_l^\gamma/Z,\R,z}}}e_{\gamma,u}(z,V)\Tr_s\Bigl[\gamma^\mathcal{E}N_uI_u(z,V)\Bigr]\,\dvol_{g^{N_{Z^\gamma/Z,\R}}}(V)\dvol_{g^{T_\R Z^\gamma}}(z),\\
\nu_{\gamma,u,e^{ip\theta}}:&=\sum_{l\in L_0}e^{ip\theta}_{|_{Z_l^\gamma}}\cdot\underline{\nu}_{\gamma,u,l},
\end{align*}
and let $\theta_{\gamma,e^{ip\theta}}(s)$ be the Mellin transform of $u\mapsto\nu_{\gamma,u,e^{ip\theta}}$, 
\[\theta_{\gamma,e^{ip\theta}}(s)=-\frac{1}{\Gamma(s)}\int_0^\infty\nu_{\gamma,u,e^{ip\theta}}u^{s-1}\,du.\]
By (\ref{äq3}), we see
\begin{align}\label{proof thm 1 d}
\theta_{\gamma,e^{ip\theta}}'(0)=&-\int_0^1\Bigl(\nu_{\gamma,u,e^{ip\theta}}-\sum_{j=-d-1}^0b_{\gamma,j,e^{ip\theta}}u^j
\Bigr)\,\frac{du}{u}\nonumber\\
&-\int_1^\infty\nu_{\gamma,u,e^{ip\theta}}\,\frac{du}{u}-\sum_{j=-d-1}^{-1}\frac{b_{\gamma,j,e^{ip\theta}}}{j}+\Gamma'(1)b_{\gamma,0,e^{ip\theta}}.
\end{align}
From Lemma \ref{äqkern}, Lemma \ref{äqcor}, Lemma \ref{large param} c) and (\ref{proof thm 1 a})-(\ref{proof thm 1 d}), we get
\[\psi_{1/\sqrt{p}}\zeta'_{\gamma,p}(0)=\log(p)p^{n_\gamma}\vp b_{\gamma,0,e^{ip\theta}}+p^{n_\gamma}\rk(\mathcal{E})\vp\theta_{\gamma,e^{ip\theta}}'(0)+o(p^{n_\gamma}).\]
To end the proof we will need the following:
\begin{Lemma}[{\cite[Lemma 3.27]{P23}}]\label{Ascoli}
Let $X$ be a compact manifold and $(E,h^E)$ a Hermitian vector bundle on $X$ with connection $\nabla^E$. Let $\{f_n\}_n\subset\Gamma(X,E)$ be a sequence converging weakly to some distribution $f$. If for any $k\in\N_0$ there is $C_k>0$ such that $\sup_n\|f_n\|_{C^k(X,E)}\le C_k$, then $f$ is smooth and $f_n$ converges in the $C^\infty$ topology to $f$.
\end{Lemma}
\noindent If we define $\underline{\theta}_{\gamma,l}(s)$ to be the Mellin transform of $u\mapsto\underline{\nu}_{\gamma,u,l}$, the statement then follows from Lemma \ref{Ascoli} with
\begin{align*}
&\alpha_{\gamma,l}:=\vp \underline{b}_{\gamma,0,l}\qquad\text{and}\qquad\beta_{\gamma,l}:=\rk(\mathcal{E})\vp\underline{\theta}_{\gamma,l}'(0).\qedhere
\end{align*}
\end{proof}

\phantomsection
\addcontentsline{toc}{section}{References} 
\bibliographystyle{tessmeramsalpha}
{\small\bibliography{Literatur}}

\vspace{0.5em}
%Düsseldorf University, Germany
 
\textit{Email address:} \href{mailto:pascal\_tessmer@yahoo.de}{\color{black}{\texttt{pascal\_tessmer@yahoo.de}}}

\end{document}